\documentclass[11pt]{amsart}

\usepackage{latexsym} 
\usepackage{amsmath} 
\usepackage{epsfig}
\usepackage{amssymb}
\usepackage{enumerate}
\usepackage{color}
\usepackage{graphicx}
\usepackage{hyperref}

\newtheorem{theorem}{Theorem}[section]
\newtheorem{lemma}[theorem]{Lemma}
\newtheorem{corollary}[theorem]{Corollary}

\newtheorem{observation}[theorem]{Observation}
\newtheorem{sublemma}{}[theorem]

\newcommand{\cC}{{\mathcal C}}

\newcommand{\cG}{{\mathcal G}}

\newcommand{\cJ}{{\mathcal J}}

\newcommand{\cM}{{\mathcal M}}
\newcommand{\cN}{{\mathcal N}}

\newcommand{\cT}{{\mathcal T}}

\makeatletter
\renewcommand\subsection{\@startsection{subsection}{2}%
  \z@{.5\linespacing\@plus.7\linespacing}{.5em}%
  {\normalfont\scshape}}
\makeatother

\title[Displaying trees across two networks]{Displaying trees across two phylogenetic networks}

\author{Janosch D\"ocker, Simone Linz, and Charles Semple} 
\thanks{We thank Britta Dorn for insightful discussions. The second and third author thank the New Zealand Marsden Fund for their financial support.}

\address{Department of Computer Science, University of T\"ubingen, T\"ubingen, Germany}
\email{janosch.doecker@uni-tuebingen.de}

\address{School of Computer Science, University of Auckland, Auckland, New Zealand}
\email{s.linz@auckland.ac.nz}
\address{School of Mathematics and Statistics, University of Canterbury, Christchurch, New Zealand}
\email{charles.semple@canterbury.ac.nz}

\date{\today}
 
\begin{document}

\begin{abstract}
Phylogenetic networks are a generalization of phylogenetic trees to leaf-labeled directed acyclic graphs that represent ancestral relationships between species whose past includes non-tree-like events such as hybridization and horizontal gene transfer. Indeed, each phylogenetic network embeds a collection of phylogenetic trees. Referring to the collection of trees that a given phylogenetic network $\cN$ embeds as the display set of $\cN$, several questions in the context of the display set of $\cN$ have recently been analyzed. For example, the widely studied {\sc Tree-Containment} problem asks if a given phylogenetic tree is contained in the display set of a given network. The focus of this paper are two questions that naturally arise in comparing the display sets of two phylogenetic networks. First, we analyze the problem of deciding if the display sets of two phylogenetic networks have a tree in common. Surprisingly, this problem turns out to be NP-complete even for two temporal normal networks. Second, we investigate the question of whether or not the display sets of two phylogenetic networks are equal. While we recently showed that this problem is polynomial-time solvable for a normal and a tree-child network, it is computationally hard in the general case. In establishing hardness, we show that the problem is contained in the second level of the polynomial-time hierarchy. Specifically, it is $\Pi_2^P$-complete. Along the way, we show that two other problems are also  $\Pi_2^P$-complete, one of which being a generalization of {\sc Tree-Containment}.
\end{abstract}

\maketitle

\section{Introduction}
In trying to disentangle the evolutionary history of species, phylogenetic networks, which are leaf-labeled directed acyclic graphs, are becoming increasingly important. From a biological as well as from a mathematical viewpoint, phylogenetic networks are often regarded as a tool to summarize a collection of conflicting phylogenetic trees. Due to processes such as hybridization and lateral gene transfer, the evolution at the species-level is not necessarily tree-like. Nevertheless, individual genes or parts thereof are usually assumed to evolve in a tree-like way. It is consequently of interest to construct phylogenetic networks that embed a collection of phylogenetic trees or, reversely, summarize the phylogenetic trees that are embedded in a given phylogenetic network. These and related types of problems have recently attracted considerable attention from the mathematical community as they lead to a number of challenging questions. One of the most studied questions in this context is called {\sc Tree-Containment}. Given a phylogenetic network $\cN$ and a  phylogenetic tree $\cT$, this problem asks whether or not $\cN$ embeds $\cT$. While {\sc Tree-Containment} is NP-complete in general~\cite{kanj08}, it has been shown to be polynomial-time solvable for several popular classes of phylogenetic networks, e.g. so-called tree-child and reticulation-visible networks~\cite{bordewich16,gunawan17,iersel10}. Currently, the fastest algorithm that solves {\sc Tree-Containment} for these latter types of networks has a running time that is linear in the size of $\cN$ and, hence, linear in the number of leaves of $\cN$~\cite{weller18}.

Pushing {\sc Tree-Containment} into a novel direction, Gunawan et al.~\cite{gunawan17} have recently posed the question of how one can check if two reticulation-visible networks embed the same set of phylogenetic trees. Since the number of trees that a phylogenetic network $\cN$ embeds grows exponentially with the number $k$ of vertices in $\cN$ whose in-degree is at least two, there is no immediate check that can be performed in polynomial time.  In particular, the number of phylogenetic trees that $\cN$ embeds is bounded above by $2^k$, and it was shown independently in~\cite[Theorem 1]{iersel10} and~\cite[Corollary 3.4]{willson12} that this upper bound is sharp for the class of normal networks. 

Referring to the collection of phylogenetic trees that a given phylogenetic network embeds as its {\it display set} (formally defined in Section~\ref{sec:prelim}), we investigate two questions that naturally arise in comparing the display sets of two phylogenetic networks. The first question asks if the display sets of two phylogenetic networks have a common element. We call this problem {\sc Common-Tree-Containment} and show in Section~\ref{sec:CTC} that it is NP-complete even when the two input networks are both temporal and normal. Strikingly, the class of temporal and normal networks is a strict subclass of the class of tree-child and, hence, reticulation-visible networks for which {\sc Tree-Containment} is polynomial-time solvable. The second problem, which we refer to as {\sc Display-Set-Equivalence}, is the problem of Gunawan et al.~\cite{gunawan17} mentioned above that asks, without restricting to a particular class of phylogenetic networks, if the display sets of two networks are equal. While we recently showed that this problem has a polynomial-time algorithm for when the input consists of a normal and a tree-child network~\cite{doecker}, we show in Section~\ref{sec:DSE} that the problem is computationally hard for two arbitrary phylogenetic networks. Specifically, we show that {\sc Display-Set-Equivalence} is $\Pi_2^P$-complete or, in other words, complete for the second level of the polynomial-time hierarchy\cite{stockmeyer76}. Intuitively, this problem is therefore much harder to solve than any NP-complete or co-NP-complete problem. In establishing the result, we also show that deciding if the display set of one phylogenetic network is contained in the display set of another network is $\Pi_2^P$-complete.

The paper is organized as follows. The next section contains preliminaries that are used throughout the paper, formal statements of the decision problems that are mentioned in the previous paragraph, and some relevant details about the polynomial-time hierarchy. Section~\ref{sec:CTC} establishes NP-completeness of {\sc Common-Tree-Containment} and Section~\ref{sec:DSE} establishes $\Pi_2^P$-completeness of {\sc Display-Set-Equivalence}. Lastly, Section~\ref{sec:conclu} contains some concluding remarks and highlights three corollaries that follow from the  results in Sections~\ref{sec:CTC}.

\section{Preliminaries}\label{sec:prelim}
This section provides notation and terminology that is used in the remaining sections. Throughout this paper, $X$ denotes a non-empty finite set. Let $G$ be a directed acyclic graph. For two distinct vertices $u$ and $v$ in $G$, we say that $u$ is an {\it ancestor} of $v$ and $v$ is a {\it descendant} of $u$, if there is a directed path from $u$ to $v$ in $G$. If  $(u,v)$ is an edge in $G$, then $u$ is a {\it parent} of $v$ and $v$ is a {\it child} of $u$. Moreover, a vertex of $G$ with in-degree one and out-degree zero is a {\it leaf} of $G$.

\noindent {\bf Phylogenetic networks and trees.} A {\it rooted binary phylogenetic network $\cN$ on $X$} is a (simple) rooted acyclic digraph that satisfies the following properties:
\begin{enumerate}[(i)]
\item the (unique) root has out-degree two,
\item the set $X$ is the set of vertices of out-degree zero, each of which has in-degree one, and
\item all other vertices have either in-degree one and out-degree two, or in-degree two and out-degree one.
\end{enumerate}
The set $X$ is the {\it leaf set} of $\cN$. Furthermore, the vertices of in-degree one and out-degree two are {\it tree vertices}, while the vertices of in-degree two and out-degree one are {\it reticulations}. An edge directed into a reticulation is called a {\it reticulation edge} while each non-reticulation edge is called a {\it tree edge}. 

Let $\cN$ be a rooted binary phylogenetic network on $X$. If $\cN$ has no reticulations, then $\cN$ is said to be a {\it rooted binary phylogenetic $X$-tree}. To ease reading and since all phylogenetic networks considered in this paper are rooted and binary, we refer to a rooted binary phylogenetic network (resp. a rooted binary phylogenetic tree) simply as a {\it phylogenetic network} (resp. {\it a phylogenetic tree}). 

Now let $\cT$ be a phylogenetic $X$-tree. If $Y=\{y_1, y_2, \ldots, y_m\}$ is a subset of $X$, then $\cT[-y_1,y_2,\ldots,y_m]$ and, equivalently, $\cT|(X-Y)$ denote the phylogenetic tree with leaf set $X-Y$ that is obtained from the minimal rooted subtree of $\cT$ that connects all leaves in $X-Y$ by suppressing all vertices of in-degree one and out-degree one.

\noindent {\it Remark.} Throughout the paper, we frequently detail constructions of phylogenetic networks. To this end, we sometimes need labels of internal vertices. Their only purpose is to make references. Indeed, they should not be regarded as genuine labels as those used for the leaves of a phylogenetic network.

\noindent {\bf Classes of phylogenetic networks.} Let $\cN$ be a phylogenetic network on $X$ with vertex set $V$. An edge $e=(u,v)$ is a {\it shortcut} if there is a directed path from $u$ to $v$ whose set of edges does not contain $e$. A vertex $v$ of $\cN$ is called {\it visible} if there exists a leaf $\ell\in X$ such that each directed path from the root of $\cN$ to $\ell$ passes through $v$. Now $\cN$ is {\it reticulation-visible} if each reticulation in $\cN$ is visible, and  $\cN$ is {\it tree-child} if each non-leaf vertex in $\cN$ has a child that is a leaf or a tree vertex. Lastly, $\cN$ is {\it normal} if it is tree-child and does not contain any shortcuts. Clearly, by definition, each normal network is also tree-child. Furthermore, it follows from the next well-known equivalence result~\cite{cardona09} that each tree-child network is also reticulation-visible.

\begin{lemma}
Let $\cN$ be a phylogenetic network. Then $\cN$ is tree-child if and only if each vertex of $\cN$ is visible.
\end{lemma}

\noindent Thus, the class of normal networks is a subclass of tree-child networks. Furthermore, if there exists a map $t:V\rightarrow {\mathbb R}^{+}$ that assigns a time stamp to each vertex of $\cN$ and satisfies the following two properties:
\begin{enumerate}[(i)]
\item $t(u)=t(v)$ whenever $(u,v)$ is a reticulation edge and 
\item  $t(u)<t(v)$ whenever $(u,v)$ is a tree edge,
\end{enumerate}
then we say that $\cN$ is {\it temporal}, in which case we call $t$ a {\it temporal labeling} of $\cN$. Note that, although normal networks have no shortcuts, a normal network need not be temporal.
Tree-child, normal, and temporal networks were first introduced by Cardona et al.~\cite{cardona09}, Willson~\cite{willson10}, and Moret et al.~\cite{moret04}, respectively. 

\noindent {\bf Caterpillars.} Let $\cC$ be a phylogenetic tree with leaf set $\{\ell_1,\ell_2,\ldots,\ell_n\}$. Furthermore, for each $i\in\{1,2,\ldots,n\}$ let $p_i$ denote the parent of $\ell_i$. Then $\cC$ is called a {\it caterpillar} if $n\geq 2$ and the elements in the leaf set of $\cC$ can be ordered, say $\ell_1, \ell_2, \ldots, \ell_n$, so that $p_1=p_2$ and, for all $i\in \{3, 4, \ldots, n\}$, we have $(p_i, p_{i-1})$ as an edge in $\cC$. In this case, we denote $\cC$ by $(\ell_1, \ell_2, \ldots, \ell_n)$. Additionally, we say that a phylogenetic $X$-tree $\cT$ {\it contains a caterpillar $\cC=(\ell_1,\ell_2,\ldots,\ell_n)$} if $\cT$ has a subtree that is a subdivision of $\cC$.

\begin{figure}[t]
\center
\includegraphics[width=.8\textwidth]{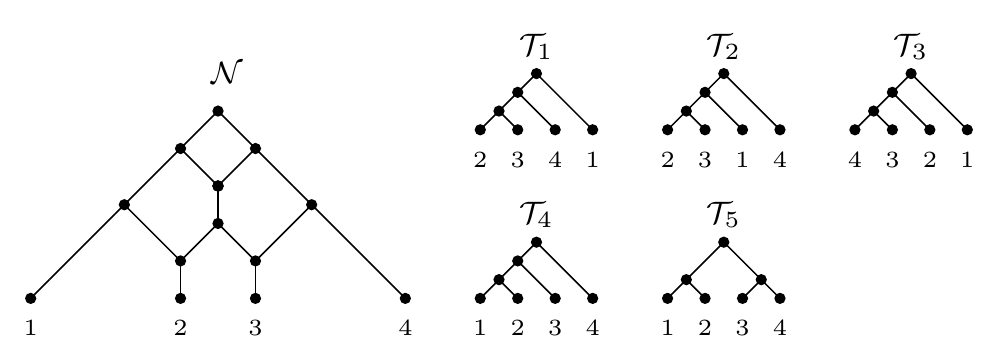}
\caption{A phylogenetic network $\cN$ and the display set of $\cN$ that consists of the five trees shown on the right-hand side.}
\label{fig:display}
\end{figure}

\noindent {\bf Displaying.} Let $\cN$ be a phylogenetic network on $X$ and let $\cT$ be a phylogenetic $Y$-tree such that $Y\subseteq X$. 
Then $\cN$ {\it displays} $\cT$ if, up to suppressing vertices of in-degree one and out-degree one, $\cT$ can be obtained from $\cN$ by deleting edges and vertices, in which case, the edge set, denoted by $E_\cT$, of the resulting acyclic directed graph is called an  {\it embedding} of $\cT$ in $\cN$. If $\cN$ displays $\cT$, note that the root of an embedding of $\cT$ in $\cN$ does not necessarily coincide with the root of $\cN$. In fact, throughout this paper, we impose that the root of an embedding has in-degree zero and out-degree two.
Moreover, the {\it display set} of $\cN$, denoted by $T(\cN)$, consists of all phylogenetic $X$-trees that are displayed by $\cN$. As mentioned in the introduction, the size of $T(\cN)$ is bounded above by $2^k$, where $k$ is the number of reticulations in $\cN$. To illustrate, Figure~\ref{fig:display} shows a phylogenetic network $\cN$ with $T(\cN)=\{\cT_1,\cT_2,\ldots,\cT_5\}$, where the five trees in $T(\cN)$ are shown on the right-hand side of the same figure. In this as well as in all other figures throughout the paper, edges are directed downwards.

Again, let $\cN$ be a phylogenetic network on $X$, and let $S$ be a subset of the edges of $\cN$. Then $S$ is a {\it switching} of $\cN$ if, for each reticulation $v$ of $\cN$, $S$ contains precisely one of the two reticulation edges that are directed into $v$. Now, let $S$ be a switching of $\cN$. If we delete each reticulation edge in $\cN$ that is not in $S$ and, repeatedly, suppress each resulting vertex with in-degree one and out-degree one, delete each vertex with in-degree one and out-degree zero that is not in $X$, and delete each vertex with in-degree zero and out-degree one, we obtain a phylogenetic $X$-tree $\cT$, in which case, we say that $S$ {\em yields} $\cT$. Note that $\cT$ is displayed by $\cN$. Conversely, observe that, if $\cT$ is a phylogenetic $X$-tree that is displayed by $\cN$ , then there exists a switching of $\cN$ that yields $\cT$. We summarize this in the following observation.

\begin{observation}
A phylogenetic network $\cN$ on $X$ displays a phylogenetic $X$-tree $\cT$ if and only if there exists a switching of $\cN$ that yields $\cT$.
\end{observation}

\noindent{\bf Problem statements.} {\sc Tree-Containment} is a well known problem in the study of phylogenetic networks and its computational complexity has extensively been analyzed for various  network classes. In the language of this paper, it can be stated as follows.

\noindent {\sc Tree-Containment}\\
\noindent{\bf Input.} A phylogenetic $X$-tree $\cT$ and phylogenetic network $\cN$ on $X$.\\
\noindent{\bf Question.} Is $\cT\in T(\cN)$?\\

\noindent While {\sc Tree-Containment} is concerned with a single display set, it is natural to compare display sets across phylogenetic networks, e.g. in the context of comparing networks. To make a first step in this direction, the focus of this paper are the following three decision problems that compare the display sets of two phylogenetic networks.

\noindent {\sc Common-Tree-Containment}\\
\noindent{\bf Input.} Two phylogenetic networks $\cN$ and $\cN'$ on $X$.\\
\noindent{\bf Question.} Is $T(\cN)\cap T(\cN')\ne\emptyset$?\\

\noindent {\sc Display-Set-Containment}\\
\noindent{\bf Input.} Two phylogenetic networks $\cN$ and $\cN'$ on $X$.\\
\noindent{\bf Question.} Is $T(\cN)\subseteq T(\cN')$?\\

\noindent {\sc Display-Set-Equivalence}\\
\noindent{\bf Input.} Two phylogenetic networks $\cN$ and $\cN'$ on $X$.\\
\noindent{\bf Question.} Is $T(\cN)=T(\cN')$?\\

\noindent We note that  {\sc Tree-Containment} is a special case of both {\sc Display-Set-Con\-tainment} and  {\sc Common-Tree-Containment}. Hence, NP-hardness of the two latter problems follows immediately for when $\cN$ and $\cN'$ are two arbitrary phylogenetic networks. Nevertheless, as we will see in Sections~\ref{sec:CTC} and~\ref{sec:DSE}, we pinpoint the complexity of {\sc Common-Tree-Containment} and  {\sc Display-Set-Containment} exactly. In particular, we will show that (i) {\sc Common-Tree-Containment} is NP-complete even for when $\cN$ and $\cN'$ are both temporal and normal and (ii) {\sc Display-Set-Containment} is complete for the second level of the polynomial-time hierarchy. This last result turns out to be a key ingredient in showing that {\sc Display-Set-Equivalence} is also complete for the second level of the polynomial-time hierarchy.

\noindent {\bf The polynomial hierarchy.} The {\it polynomial-time hierarchy} (or short, {\it polynomial hierarchy})~\cite{garey79,stockmeyer76} consists of a system of complexity classes that are defined recursively and generalize the classes P, NP, and co-NP. In particular, for any integer $k\geq 0$, referred to as {\it level}, we have
\[\Sigma_0^P = \Pi_0^P = \text{P},\]
\[\Sigma_{k+1}^P = \text{NP}^{\Sigma_k^P}\text{\hspace{3mm}and\hspace{3mm}}\Pi_{k+1}^P = \text{co-NP}^{\Sigma_k^P}.\]
Level-0 of the hierarchy coincides with the class P (i.e. $\Sigma_0^P=\Pi_0^P$) while level-1 coincides with the class NP (i.e. $\Sigma_1^P$) and co-NP (i.e. $\Pi_1^P$), respectively.  
For all $k \geq 0$, it is an open problem whether or not $\Sigma_k^P \neq \Sigma_{k+1}^P$. Specifically, for $k = 0$, this is the fundamental P versus NP problem. If $\Sigma_k^P = \Sigma_{k+1}^P$ or $\Pi_k^P = \Pi_{k+1}^P$ for some $k\geq 0$, then this would result in a collapse of the polynomial hierarchy to the $k$-th level.

In Section~\ref{sec:DSE}, we show that {\sc Display-Set-Containment} and {\sc Display-Set-Equivalence}  are both $\Pi_2^P$-complete. Intuitively, problems that are complete for the second level of the polynomial hierarchy are more difficult than problems that are complete for the first level. Recall that a decision problem is in co-NP if a no-instance can be verified in polynomial time given an appropriate certificate. Now, similar to showing that a problem is co-NP-complete, a proof that establishes $\Pi_2^P$-completeness consists of two steps: (i) show that a problem is in $\Pi_2^P$, and (ii) establish a polynomial-time reduction from a problem that is known to be $\Pi_2^P$-complete to the problem at hand. With regards to (i), a decision problem is {\it in $\Pi_2^P$} if a no-instance can be verified in polynomial time when one is given an appropriate certificate and has access to an NP-{\it oracle}, that is, an oracle that can solve NP-complete problems in constant time.

\section{Hardness of {\sc Common-Tree-Containment}}\label{sec:CTC}

As noted in the introduction, {\sc Tree-Containment} is NP-complete in general, but polynomial-time solvable for several popular classes of phylogenetic networks such as tree-child and reticulation-visible networks. In this section, we show that no such dichotomy holds for {\sc Common-Tree-Containment}. 
In particular, we will show that this problem is NP-complete even if the input consists of two temporal normal networks.
To establish the result, we use a reduction from the classical computational problem 3-SAT.

\noindent {\sc 3-SAT}\\
\noindent{\bf Input.} A set $V=\{v_1,v_2,\ldots,v_n\}$ of variables, and a set $\{C_1,C_2,\ldots,C_m\}$ of clauses such that each clause is a disjunction of exactly three  literals and each literal is an element in $\{v_i,\bar{v}_i:i\in\{1,2,\ldots,n\}\}$.\\
\noindent{\bf Question.} Does there exist a truth assignment for $V$ that satisfies each clause $C_j$ with $j\in\{1,2,\ldots,m\}$?\

Let $I$ be an instance of {\sc 3-SAT}, and let $C_j=(x_j^1\vee x_j^2\vee x_j^3)$ be a clause of $I$ for $j\in\{1,2,\ldots,m\}$. Then, for some indices $k$, $k'$, and $k''$ in $\{1,2,\ldots,n\}$, we have $x_j^1\in\{v_k,\bar{v}_k\}$, $x_j^2\in\{v_{k'},\bar{v}_{k'}\}$, and $x_j^3\in\{v_{k''},\bar{v}_{k''}\}$. 
Without loss of generality, we impose the following two restrictions on $I$:
\begin{enumerate}
\item [(R1)] for each $v_i\in V$ with $i\in\{1,2,\ldots,n\}$, at most one element in $\{v_i,\bar{v}_i\}$ is a literal of $C_j$ and
\item [(R2)] $k<k'<k''$.
\end{enumerate}
Now, for each clause $C_j$, we construct the two clause gadgets $G_j^A$ and $G_j^B$ that are shown in Figure~\ref{fig:CTC-clause-gadgets}. We next establish a simple lemma.

\begin{figure}[t]
\center
\includegraphics[width=.75\textwidth]{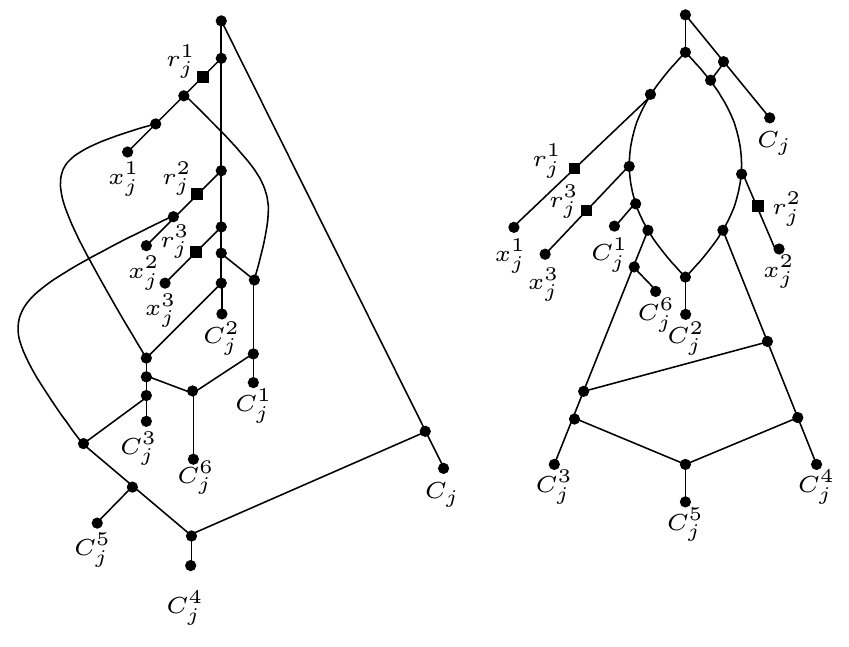}
\caption{For a clause $C_j=(x_j^1\vee x_j^2\vee x_j^3)$, the clause gadget $G_j^A$ (left) and the clause gadget $G_j^B$ (right). Leaves are bijectively labeled with the elements in $\{C_j,C_j^1,C_j^2,\ldots,C_j^6,x_j^1,x_j^2,x_j^3\}$. Furthermore, each gadget has three vertices of in-degree one and out-degree one indicated by small squares labeled $r_j^1$, $r_j^2$, and $r_j^3$.}
\label{fig:CTC-clause-gadgets}
\end{figure}

\begin{lemma}\label{lem:clause-gadget}
Let $G_j^A$ and $G_j^B$ be the two clause gadgets that are shown in Figure~\ref{fig:CTC-clause-gadgets}. Obtain two phylogenetic networks $\cG_j^A$ and $\cG_j^B$ from $G_j^A$ and $G_j^B$, respectively, by suppressing the three vertices  $r_j^1$, $r_j^2$, and $r_j^3$ of in-degree one and out-degree one. Then $T(\cG_j^A)\cap T(\cG_b^B)=\emptyset$.
\end{lemma}

\begin{proof}
To see that $T(\cG_j^A)\cap T(\cG_j^B)=\emptyset$, observe that each tree in $T(\cG_j^A)$ contains the caterpillar $(x_j^2,x_j^3,x_j^1)$, whereas each tree in $T(\cG_j^B)$ contains the caterpillar $(x_j^1,x_j^3,x_j^2)$.
\end{proof}

Let $S=(s_1,s_2,\ldots,s_n)$ be an arbitrary tuple, and let $r$ be an element that is not contained in $S$. We write $(r)||S$ to denote the tuple $(r,s_1,s_2,\ldots,s_n)$  obtained by concatenating $r$ and $S$. With this definition in hand, we are now in a position to establish the main result of this section.

\begin{theorem}\label{t:CTC}
{\sc Common-Tree-Containment} is \emph{NP}-complete when the input consists of two temporal normal networks.
\end{theorem}

\begin{proof}
For two normal networks, van Iersel et al.~\cite{iersel10} showed that the running time of {\sc Tree-Containment} is polynomial in the size of this leaf set. Hence, it follows that {\sc Common-Tree-Containment} is in NP for two normal networks.

Let $I$ be an instance of 3-SAT with $n$ variables and $m$ clauses. Using the same notation as in the formal statement of 3-SAT, we construct
two phylogenetic networks $\cN$ and $\cN'$ on
\begin{eqnarray*}
X&=&\{C_j,C_j^1,C_j^2,\ldots,C_j^6,x_j^1,x_j^2,x_j^3:j\in\{1,2,\ldots,m\}\}\cup\\
&&\{v_i:i\in\{1,2,\ldots,n\}\}
\end{eqnarray*}
as follows. Let $\cT$  be the phylogenetic tree obtained by creating a vertex $\rho$, adding an edge that joins $\rho$ with the root of the caterpillar $(v_1,v_2,\ldots,v_n)$, and adding an edge that joins $\rho$ with the root of the caterpillar $(c_1,c_2,\ldots,c_m)$. 
Now, setting $\cM=\cM'=\cT$, let $\cN$ and $\cN'$ be the two phylogenetic networks obtained from $\cM$ and $\cM'$, respectively, by applying the following four-step process.
\begin{enumerate}
\item For all $j\in\{1,2,\ldots,m\}$, replace $c_j$ with $G_j^A$ in $\cM$ and replace $c_j$ with $G_j^B$ in $\cM'$.
\item For all $i\in\{1,2,\ldots,n\}$, subdivide the edge directed into $v_i$ with a new vertex $d_i$ in $\cM$ and $\cM'$.
\item \label{step} For each $j\in\{1,2,\ldots,m\}$ in increasing order, consider $C_j=(x_j^1\vee x_j^2\vee x_j^3)$. Let $v_{k_\ell}$ be the unique element in $V$ such that $x_j^\ell\in\{v_{k_\ell},\bar{v}_{k_\ell}\}$ for each $\ell\in\{1,2,3\}$.
If $x_j^\ell=v_{k_\ell}$, subdivide the edge directed into $v_{k_\ell}$ with a new vertex $u_j^\ell$ in $\cM$ and subdivide the edge directed into $d_{k_\ell}$ with a new vertex $u_j^\ell$ in $\cM'$. Otherwise, subdivide the edge directed into $d_{k_\ell}$ with a new vertex $u_j^\ell$ in $\cM$ and subdivide the edge directed into $v_{k_\ell}$ with a new vertex $u_j^\ell$ in $\cM'$. Add a new edge $(u_j^\ell,r_j^\ell)$ in $\cM$ and $\cM'$.
\begin{figure}[t]
\center
\includegraphics[width=\textwidth]{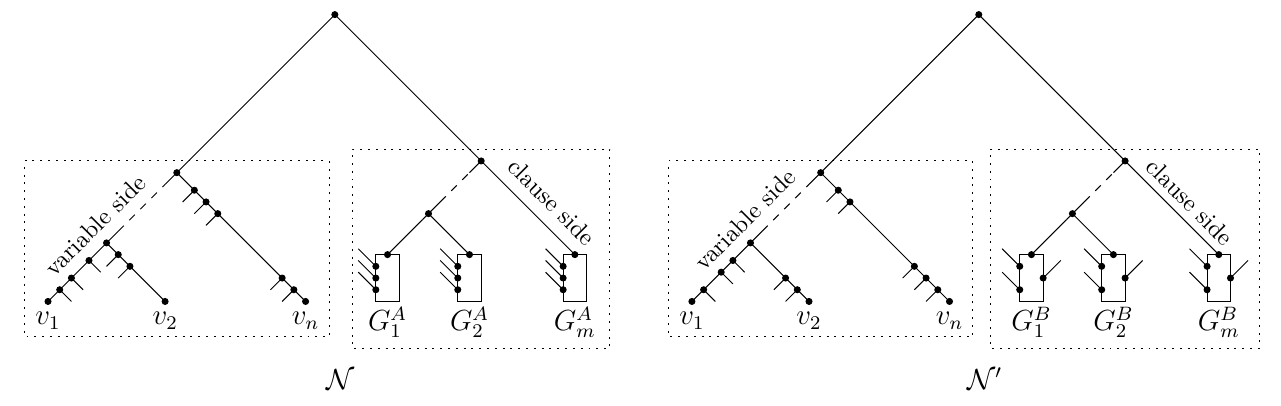}
\caption{Overview of the construction of the two temporal normal networks $\cN$ and $\cN'$ in the proof of Theorem~\ref{t:CTC}. Dangling edges on the clause and variable side of $\cN$ and $\cN'$, respectively, are paired up depending on $I$. For details, see Step (\ref{step}) of the construction.}
\label{fig:CTC-construction}
\end{figure}
\item For each $i\in\{1,2,\ldots,n\}$, suppress the vertex $d_i$ of in-degree one and out-degree one in $\cM$ and $\cM'$.
\end{enumerate}
To illustrate, Figure~\ref{fig:CTC-construction} gives a high-level overview of the construction of $\cN$ and $\cN'$.  Observe that, for each $j\in\{1,2,\ldots,m\}$, the three vertices $r_j^1$, $r_j^2$, and $r_j^3$ in $\cN$ and $\cN'$ are reticulations. 

We next show that $\cN$ and $\cN'$ are both temporal and normal.

\begin{sublemma}\label{subl:one}
Both $\cN$ and $\cN'$ are temporal and normal. 
\end{sublemma}

\begin{proof}
We first show that $\cN$ is temporal and normal. Let \[V_r=\{r_j^\ell:j\in\{1,2,\ldots,m\}\textnormal{ and }\ell\in\{1,2,3\}\}.\] Furthermore, for each $i\in\{1,2,\ldots,n\}$, let $V_i$ consist of all vertices that lie on the unique directed path from the root of $\cN$ to $v_i$, and let \[V_v=\bigcup_{i=1}^n V_i.\] 
We begin by assigning a positive real-valued labeling $t$ to each vertex in $V_v\cup V_r$ as follows.
First, under $t$, each vertex in $V_v$ is assigned a labeling such that the following two properties are satisfied.
\begin{enumerate}[(i)]
\item If $u,v\in V_v$ and $u$ is an ancestor of $v$, then $t(u)<t(v)$.
\item For all $i\in\{1,2,\ldots,n-1\}$, the temporal labeling of each vertex in $V_i$ that is not contained in $V_{i+1}$  is smaller than the minimum temporal labeling over all vertices that are contained in $V_{i+1}$ and not in $V_{i}$ .
\end{enumerate}
By construction of $\cN$, note that such a labeling always exists.
Second, under $t$, each vertex in $V_r$ is assigned the same labeling as its unique parent that is contained in $V_v$.
Because of restrictions (R1) and (R2) that we have imposed on $I$ and the way we have assigned temporal labelings to the vertices in $V_v$, we have \[t(r_j^1)<t(r_j^2)<t(r_j^3)\] for each $j\in\{1,2,\ldots,m\}$. A routine check now shows that $t$ can be extended to a temporal labeling of $\cN$ and, thus, $\cN$ is temporal. 

Now, since $\cN$ is temporal, it follows that $\cN$ has no shortcuts. Hence, to show that $\cN$ is normal, it suffices to show that $\cN$ is tree-child. It is straightforward to check that $\cN$ has no edge $(u,v)$ such that $u$ and $v$ are both reticulations. Hence, each reticulation in $\cN$ has a child that is a tree vertex or a leaf. Furthermore, by construction, each tree vertex of $\cN$ that is a vertex of some $G_j^A$ with $j\in\{1,2,\ldots,m\}$ has a child that is a tree vertex or a leaf. Lastly, for each non-leaf vertex $v$ of $\cN$ that is neither a reticulation nor a vertex of some $G_j^A$, consider a directed path $P$ from $v$ to an element in $\{v_1,v_2,\ldots,v_n,C_1,C_2,\ldots,C_m\}$. By construction, $P$ exists. It is now easily seen that the second vertex of $P$ is a child of $v$ that is either a tree vertex or a leaf. This establishes that $\cN$ is normal. An analogous argument that uses $G_j^B$ instead of $G_j^A$ can be used to show that $\cN'$ is temporal and normal, thereby completing the proof of~(\ref{subl:one}).
\end{proof}

Since the number of vertices of a normal network is polynomial in the size of $X$~\cite{mcdiarmid15} and $|X|=10m+n$, it follows that $\cN$ and $\cN'$  can be constructed in time polynomial in the size of $X$.

\begin{sublemma}\label{slem:claim}
The instance $I$ is a yes-instance if and only if $T(\cN)\cap T(\cN')\ne\emptyset$.
\end{sublemma}

\begin{figure}[t]
\center
\includegraphics[width=\textwidth]{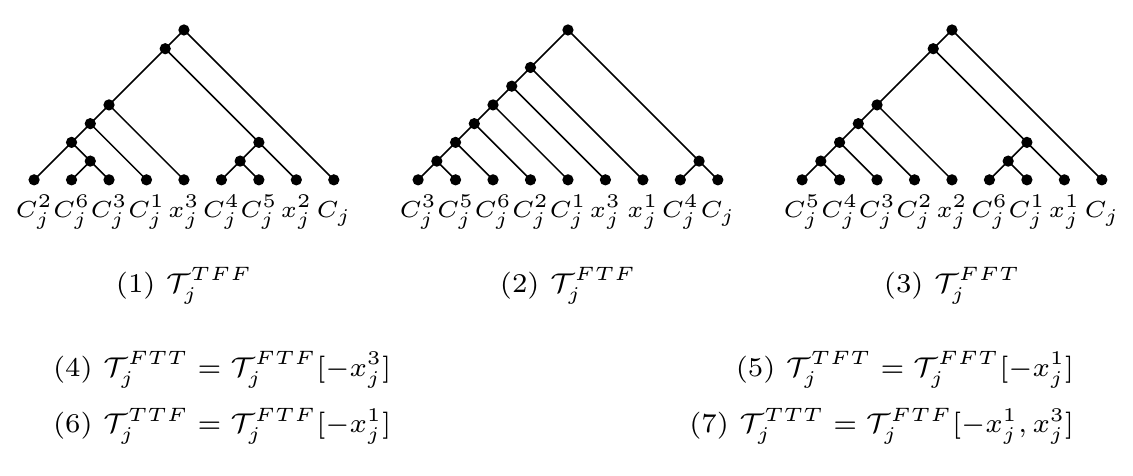}
\caption{The seven trees that are used in the proof of Theorem~\ref{t:CTC}.}
\label{fig:CTC-common-subtrees}
\end{figure}

\begin{proof}
First, suppose that $I$ is a yes-instance. We  construct a {\it variable tree} $\cT_v$ and a {\it clause tree} $\cT_c$ that, joined together, result in a phylogenetic $X$-tree that is displayed by $\cN$ and $\cN'$. Let $\beta:V\rightarrow\{F,T\}$ be a truth assignment that satisfies each clause, and let \[Y=\{x_j^\ell:j\in\{1,2,\ldots,m\}\textnormal{ and } \ell\in\{1,2,3\}\}.\] Furthermore, for each $i\in\{1,2,\ldots,n\}$, let $Y_i$ (resp. $\bar{Y}_i$) be the tuple consisting of the elements in $Y$ that equal $v_i$ (resp. $\bar{v}_i$) such that, for any two elements $x_j^\ell$ and $x_{j'}^{\ell'}$ in $Y_i$ (resp. $\bar{Y}_i$), $x_j^\ell$ precedes $x_{j'}^{\ell'}$ precisely if $j>j'$. By construction, note that the two caterpillars $(v_i)||Y_i$ and $(v_i)||\bar{Y}_i$ are displayed by $\cN$ and $\cN'$. Now, obtain $\cT_v$ from the caterpillar $(v_1,v_2,\ldots,v_n)$ by doing the following for each $i\in\{1,2,\ldots,n\}$. If $\beta(v_i)=T$, replace $v_i$ with the caterpillar $(v_i)||Y_i$; otherwise, replace $v_i$ with the caterpillar $(v_i)|| \bar{Y}_i$. Again, by construction, it is easily checked that $\cT_v$ is displayed by $\cN$ and $\cN'$. We next construct $\cT_c$. Consider a clause $C_j=(x_j^1\vee x_j^2\vee x_j^3)$. For each $\ell\in\{1,2,3\}$, set $z_\ell=T$ if $x_j^\ell$ is satisfied by $\beta$ and, otherwise, set $z_\ell=F$.  Depending on which elements in $\{z_1,z_2,z_3\}$ equal $F$ and $T$, respectively, and noting that there exists some $\ell$ for which $z_\ell=T$, we define the {\it clause tree} $\cT_j^{z_1z_2z_3}$ relative to $C_j$ to be one of the seven trees that are listed in Figure~\ref{fig:CTC-common-subtrees}.  Intuitively, $x_j^\ell$ is a leaf in $\cT_j^{z_1z_2z_3}$ precisely if $z_\ell=F$.  Now, obtain $\cT_c$ from the caterpillar $(c_1,c_2,\ldots,c_m)$ by replacing, for each $j\in\{1,2,\ldots,m\}$, the leaf $c_j$ with the clause tree relative to $C_j$. As $\cT_j^{z_1z_2z_3}$ is displayed by the two phylogenetic networks obtained from $G_j^A$ and $G_j^B$ by suppressing the three vertices $r_j^1$, $r_j^2$, and $r_j^3$ of in-degree one and out-degree one, it follows that  $\cT_j^{z_1z_2z_3}$ is also displayed by $\cN$ and $\cN'$. In turn, this implies that, by construction, $\cT_c$ is displayed by $\cN$ and $\cN'$. Lastly, we construct a phylogenetic tree $\cT$ on $X$ by creating a vertex $\rho$, adding a new edge that joins $\rho$ with the root of $\cT_v$, and a new edge that joins $\rho$ with the root of $\cT_c$. As $\cT_v$ and $\cT_c$ are displayed by $\cN$ and $\cN'$, it is easily checked that $\cT$ is displayed by $\cN$ and $\cN'$, and so $T(\cN)\cap T(\cN')\ne\emptyset$.

Second, suppose that $T(\cN)\cap T(\cN')\ne\emptyset$. Let $\cT$ be a phylogenetic $X$-tree that is displayed by $\cN$ and $\cN'$.  Furthermore, let $j,j'\in\{1,2,\ldots,m\}$, and let $\ell,\ell'\in\{1,2,3\}$.  For each reticulation $r_j^\ell$ in $\cN$ (resp. $\cN'$), we say that {\it $\cT$ picks $x_j^\ell$ from the clause side} of  $\cN$ (resp. $\cN'$) if $\cT$ has a vertex whose set of descendants contains $x_j^\ell$ and $C_j$ but does not contain any element in $V$; 
otherwise, we say that {\it $\cT$ picks $x_j^\ell$  from the variable side} of $\cN$ (resp. $\cN'$). Intuitively, $x_j^\ell$ is picked from the clause side of $\cN$ (resp. $\cN'$) precisely if the embedding of $\cT$ in $\cN$ (resp. $\cN'$) contains the reticulation edge directed into $r_j^\ell$ whose two end vertices are vertices of $G_j^A$ (resp. $G_j^B$). Note that, as $\cT$ is displayed by $\cN$ and $\cN'$, we have  that $\cT$ picks $x_j^\ell$  from the variable side of $\cN$ if and only if $\cT$ picks $x_j^\ell$  from the variable side of $\cN'$.
We next make two observations:
\begin{enumerate}
\item[(O1)] For each clause $C_j=(x_j^1\vee x_j^2\vee x_j^3)$, it follows from Lemma~\ref{lem:clause-gadget} that $\cT$ picks at most two of $x_j^1$, $x_j^2$, and  $x_j^3$ from the clause side of $\cN$ and $\cN'$.
\item[(O2)] It follows from Step~(\ref{step}) in the construction of $\cN$ and $\cN'$, and the fact that $\cT$ is displayed by $\cN$ and $\cN'$ that, if $\cT$ picks $x_j^\ell$ from the variable side of $\cN$ and $\cN'$, and $x_j^\ell=v_i$ for some $i\in\{1,2,\ldots,n\}$, then each $x_{j'}^{\ell'}$ with $x_{j'}^{\ell'}=\bar{v}_i$ is picked from the clause side of $\cN$ and $\cN'$. Similarly, if $\cT$ picks $x_j^\ell$ from the variable side of $\cN$ and $\cN'$, and $x_j^\ell=\bar{v}_i$ for some $i\in\{1,2,\ldots,n\}$, then each $x_{j'}^{\ell'}$ with $x_{j'}^{\ell'}=v_i$ is picked from the clause side of $\cN$ and $\cN'$.
\end{enumerate}
Now, let $\beta$ be the truth assignment that is defined as follows. For each $i\in\{1,2,\ldots,n\}$, we set $v_i=T$ if there exists an element $x_j^\ell$ with $x_j^\ell=v_i$ that is picked from the variable side of $\cN$ and $\cN'$. On the other hand, we set $v_i=F$ if either there exists an element $x_j^\ell$ with $x_j^\ell=\bar{v}_i$ that is picked from the variable side of $\cN$ and $\cN'$ or there is  no $x_j^\ell$ with $x_j^\ell\in\{v_i,\bar{v}_i\}$ that is picked from the variable side of $\cN$ and $\cN'$. Because of (O2), $\beta$ is well defined. Moreover, by (O1) it  follows that $\beta$ satisfies at least one literal of each clause and, hence, $I$ is a yes-instance. 
\end{proof}
This completes the proof of  Theorem~\ref{t:CTC}.
\end{proof}

The next corollary is an immediate consequence of Theorem~\ref{t:CTC}.

\begin{corollary}
Let $\cN$ and $\cN'$ be two temporal normal networks on $X$. It is  \emph{co-NP}-complete to decide if $T(\cN)\cap T(\cN')=\emptyset$.
\end{corollary}

\section{Hardness of {\sc Display-Set-Equivalence}}\label{sec:DSE}
In this section, we show that {\sc Display-Set-Equivalence} is $\Pi_2^P$-complete, that is, the problem is complete for the second level of the polynomial hierarchy. To establish this result, we use a chain of three polynomial-time reductions that are described in Subsections~\ref{sec:disjoint-paths},~\ref{sec:display-set-containment}, and~\ref{sec:display-set-equivalence}. Before detailing the reductions, we introduce two more decision problems that play an important role in this section.

Recall the (ordinary) {\sc 3-SAT} problem as introduced in Section~\ref{sec:CTC}. The input to an instance of  {\sc 3-SAT} consists of a boolean formula over a set of variables. Importantly, each variable is existentially quantified since we are asking whether or not there {\it exists} a truth assignment to each variable that satisfies each clause of the formula. In contrast, the following quantified version of  {\sc 3-SAT} has two different types of variables, i.e each variable is either existentially or universally quantified.

\noindent {\sc $\forall \exists$ 3-SAT}\\
\noindent{\bf Input.} A quantified boolean formula \[\Psi=\forall v_1 \forall v_2 \cdots \forall v_p \exists v_{p+1} \exists v_{p+2} \cdots \exists v_n \bigwedge_{j=1}^m C_j\] over a set of variables $V=\{v_1,v_2,\ldots,v_n\}$ such that each clause $C_j$ is a disjunction of exactly three  literals and each literal is an element in $\{v_i,\bar{v}_i:i\in\{1,2,\ldots,n\}\}$.\\
\noindent{\bf Question.} For each truth assignment $\beta^\forall:\{v_1, v_2, \ldots, v_p\}\rightarrow \{F,T\}$, does there exist a truth assignment $\beta^\exists:\{v_{p+1}, v_{p+2}, \ldots, v_p\}\rightarrow \{F,T\}$ such that, collectively, $\beta^\forall$ and $\beta^\exists$ satisfy each clause in $\Psi$?

\noindent It was shown in~\cite{stockmeyer76} that {\sc $\forall \exists$ 3-SAT} is $\Pi_2^P$-complete. Let $I$ be an instance of  {\sc $\forall \exists$ 3-SAT}. Note that each clause of $I$ has at least one literal that is an element in $\{x_i,\bar{x}_i:i\in\{p+1,p+2,\ldots,n\}\}$ since, otherwise, $I$ is a no-instance. Furthermore, if all variables are existentially quantified, then $I$ is an instance of the (ordinary) {\sc 3-SAT} problem.
Hence, we may assume throughout this section that $1\leq p<n$.

We next formally state a quantified version of the well-known NP-complete decision problem {\sc Directed-Disjoint-Connecting-Paths}~\cite{garey79,perl78}. Let $G$ be a directed graph with vertex set $V$, and let $\{(s_1,t_1),(s_2,t_2),\ldots,(s_k,t_k)\}$ be a collection of pairs of vertices in $V$. In what follows, we write $\pi_i$ to denote a directed path in $G$ from $s_i$ to $t_i$ with $i\in\{1,2,\ldots,k\}$. 

\noindent{\sc $\forall \exists$ Directed-Disjoint-Connecting-Paths} \\
{\bf Input.} A directed graph $G$ and two collections 
\begin{align*}
P^\forall &= \{(s_1,t_1),(s_2,t_2),\ldots,(s_p,t_p)\}, \\
P^\exists &= \{(s_{p+1},t_{p+1}),(s_{p+2},t_{p+2}),\ldots,(s_k,t_k)\}
\end{align*}
of pairs of vertices in $G$ such that $1\leq p< k$ and, for each $(s_i,t_i)\in P^\forall$, there exists a directed path from $s_i$ to $t_i$ in $G$.\\
{\bf Question.} For each set $\Pi^\forall=\{\pi_1,\pi_2,\ldots,\pi_p\}$ of directed paths, does there exist a set $\Pi^\forall\cup\{\pi_{p+1},\pi_{p+2},\ldots,\pi_k\}$ of mutually vertex-disjoint directed paths in $G$? 

\subsection{{{\sc $\forall \exists$ Directed-Disjoint-Connecting-Paths} is $\Pi_2^P$-complete}}\label{sec:disjoint-paths}

To show that {\sc $\forall \exists$ Directed-Disjoint-Connecting-Paths} is complete for the second level of the polynomial hierarchy, we use a polynomial-time reduction from {\sc $\forall \exists$ 3-SAT}. This reduction constructs a special instance  of {\sc $\forall \exists$ Directed-Disjoint-Connecting-Paths} for which the input graph is a particular type of phylogenetic network. 

Let $\cN$ be a phylogenetic network on $X$, let $S=\{s_1,s_2,\ldots,s_k\}$ and $T=\{t_1,t_2,\ldots,t_k\}$  be two disjoint subsets of the vertices of $\cN$ such that $T=X$, and let $p\in\{1,2,\ldots,k\}$. We call $\cN$ a {\it caterpillar-inducing} network with respect to $S$ if the network obtained from $\cN$ by deleting each vertex that lies on a directed path from a child of a vertex in $S$ to a leaf of $\cN$ is a caterpillar up to deleting all leaf labels. Moreover, we say that $\cN$ has the {\it two-path property relative to}  $p$ if, for each $i\in\{1,2,\ldots,p\}$, there are two directed paths, say $\pi_i$ and $\pi_i'$, from $s_i$ to $t_i$ such that the following three properties are satisfied:
\begin{enumerate}[(i)]
\item $\pi_i$ and $\pi_i'$ are the only directed paths from $s_i$ to $t_i$ in $\cN$,
\item $\pi_i$ and $\pi_i'$ only have the three vertices $s_i$, $t_i$, and the (unique) parent of $t_i$ as well as the edge directed into $t_i$ in common, and
\item no path in $\{\pi_i,\pi_i':i\in\{1,2,\ldots,p\}\}$ intersects with any path in $\{\pi_j,\pi_j':j\in\{1,2,\ldots,p\}-\{i\}\}$.
\end{enumerate}
Using the same notation as in the statement of {\sc $\forall \exists$ Directed-Disjoint-Connecting-Paths}, we now introduce a similar problem whose input graph is a phylogenetic network.

\noindent{\sc $\forall \exists$ Phylo-Directed-Disjoint-Connecting-Paths} \\
{\bf Input.} A phylogenetic network $\cN$ on $X$, two  sets $S=\{s_1,s_2,\ldots,s_k\}$ and $T=X=\{t_1,t_2,\ldots,t_k\}$ of vertices of $\cN$, and an integer $p$ with $1\leq p < k$ such that $\cN$ is caterpillar-inducing with respect to $S$ and has the  two-path property relative to $p$. Furthermore, the two collections 
\begin{align*}
P^\forall &= \{(s_1,t_1),(s_2,t_2),\ldots,(s_p,t_p)\}, \\
P^\exists &= \{(s_{p+1},t_{p+1}),(s_{p+2},t_{p+2}),\ldots,(s_k,t_k)\}
\end{align*}
of pairs of elements in $S$ and $T$.\\
{\bf Question.} For each set $\Pi^\forall=\{\pi_1,\pi_2,\ldots,\pi_p\}$ of directed paths, does there exist a set $\Pi^\forall\cup\{\pi_{p+1},\pi_{p+2},\ldots,\pi_k\}$ of mutually vertex-disjoint directed paths in $\cN$?

The next theorem establishes the $\Pi_2^P$-completeness of {\sc $\forall \exists$ Phylo-Direc\-ted-Disjoint-Con\-necting-Paths}. The reduction that we use for the proof has a flavor that is similar to that in \cite[page 86]{khuller94}.

\begin{theorem}\label{t:disjoint-paths}
The decision problem {\sc $\forall \exists$ Phylo-Directed-Disjoint-Con\-necting-Paths} is $\Pi_2^P$-complete. 
\end{theorem}

\begin{proof}
We first show that {\sc $\forall \exists$ Phylo-Directed-Disjoint-Connecting-Paths} is in $\Pi_2^P$. Using the same notation as in the formal statement of this problem, guess a set $\Pi^\forall=\{\pi_1,\pi_2,\ldots,\pi_p\}$  of directed paths in $\cN$. Since $\cN$ has the two-path property relative to $p$, the paths in $\Pi^\forall$ are mutually vertex disjoint. Next obtain the directed graph $G$ from $\cN$ by deleting all vertices that lie on a path in $\Pi^\forall$. Lastly, use an NP-oracle for the unquantified version of {\sc Directed-Disjoint-Connecting-Paths}  to decide if there exists a set $\Pi^\exists=\{\pi_{p+1},\pi_{p+2},\ldots,\pi_k\}$ of mutually vertex-disjoint directed paths in $G$. Since a given instance of {\sc $\forall \exists$ Phylo-Directed-Disjoint-Connecting-Paths} is a no-instance precisely if there exists some set $\Pi^\forall$ for which no choice of $\Pi^\exists$ results in a set $\Pi^\forall\cup\Pi^\exists$ of mutually vertex-disjoint directed paths in $\cN$, it follows that this problem is in co-NP$^{\text{NP}}=\Pi_2^P$. 

We now establish a polynomial-time reduction from the quantified {\sc 3-SAT} problem. Let $I$ be an instance of {\sc $\forall \exists$ 3-SAT} with boolean formula \[\Psi=\forall v_1 \forall v_2 \cdots \forall v_p \exists v_{p+1} \exists v_{p+2} \cdots \exists v_n \bigwedge_{j=1}^m C_j\] over a set $V=\{v_1,v_2,\ldots,v_n\}$ of variables. Throughout the proof, we use $C_j=(x_{3j-2}\vee x_{3j-1}\vee x_{3j})$ to refer to the three literals in $C_j$ for each $j\in\{1,2,\ldots,m\}$. 
Now, for each $i\in\{1,2,\ldots,n\}$, let $\mathcal{J}_i^+$ be the set that consists of the indices of the literals that are equal to $v_i$  and, similarly, let $\mathcal{J}_i^-$ be the set that consists of the indices of the literals that are equal to  $\bar{v}_i$. Without loss of generality, we may assume that $\cJ_i^+\ne\emptyset$ or $\cJ_i^-\ne\emptyset$ since, otherwise, $v_i$ can be deleted from $V$. 

For each variable $v_i$, we construct a variable gadget $G_i^v$ as follows: 
\begin{enumerate}
\item Create three vertices  $s_i^v$, $t_i^v$, and $y_i$. 
\item Create the (possibly empty) set of vertices $\bigcup_{l \in \mathcal{J}_i^+} \{p_l^{\text{in}}, p_l^{\text{out}}\}$ and construct the directed path 
\[
\pi^+_i = (s_i^v,\, p_{l_1}^{\text{in}},\, p_{l_1}^{\text{out}},\,p_{l_2}^{\text{in}},\, p_{l_2}^{\text{out}},\, \ldots,\,  p_{l_q}^{\text{in}},\, p_{l_q}^{\text{out}},\, y_i,t_i^v) 
\]
with $\{l_1, l_2, \ldots, l_q\} = \mathcal{J}_i^+$.
\item Create the (possibly empty) set of vertices $\bigcup_{k \in \mathcal{J}_i^-} \{n_k^{\text{in}}, n_k^{\text{out}}\}$ and construct the directed path 
\[
\pi^-_i = (s_i^v,\, n_{k_1}^{\text{in}},\, n_{k_1}^{\text{out}},\, n_{k_2}^{\text{in}},\, n_{k_2}^{\text{out}}, \ldots,\,  n_{k_r}^{ \text{in}},\, n_{k_r}^{\text{out}},\, y_i,t_i^v)
\]
with $\{k_1, k_2,\ldots, k_r\} = \mathcal{J}_i^-$.
\end{enumerate}

\begin{figure}
\centering
\includegraphics[width=.8\textwidth]{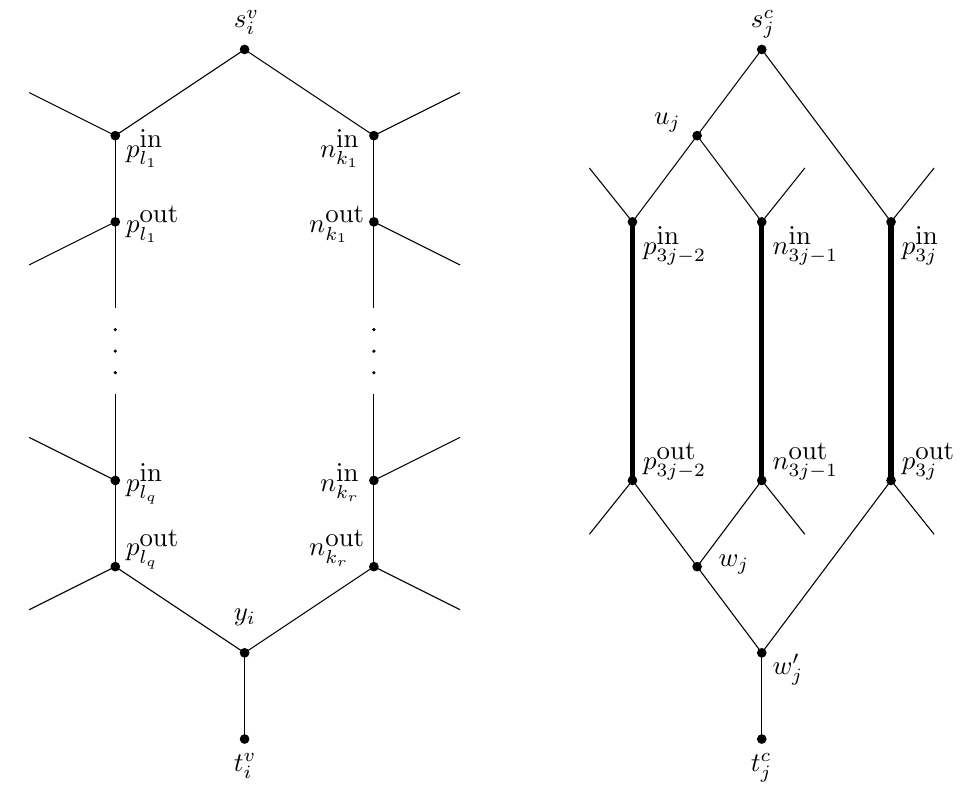}
\caption{Left: Variable gadget for a variable $v_i$. Right: Clause gadget for a clause $C_j=(x_{3j-2}\vee x_{3j-1}\vee x_{3j})$, where the second literal equals a negated variable and each of the other two literals equals an unnegated variable. The three thick edges are edges of a variable gadget. The complete construction is detailed in the proof of Theorem~\ref{t:disjoint-paths}.}
\label{fig:disjointPaths}
\end{figure}

\noindent Note that, since we do not allow for parallel edges, the last edge $(y_i,t_i^v)$ of $\pi^+_i $ and $\pi^-_i $ only appears once in $G_i^v$. Intuitively, the two paths $\pi^+_i$ and $\pi^-_i$ correspond to the two possible truth assignments for the variable~$v_i$.  To illustrate, a  generic variable gadget for $v_i$ is shown on the left-hand side of Figure \ref{fig:disjointPaths}. The additional edges in this figure that are directed into vertices of the variable gadget and directed out of vertices of this gadget will be defined as part of the clause gadget construction which we describe next.

For a clause $C_j=(x_{3j-2}\vee x_{3j-1}\vee x_{3j})$, let $i_j$, $i_j'$, and $i_j''$ be the elements in $\{1,2,\ldots,n\}$ such that $x_{3j-2}\in\{v_{i_j},\bar{v}_{i_j}\}$, $x_{3j-1}\in\{v_{i_j'},\bar{v}_{i_j'}\}$, and $x_{3j}\in\{v_{i_j''},\bar{v}_{i_j''}\}$. Now, for each $j\in\{1,2,\ldots,m\}$, add the following vertices and edges to the variable gadgets.
\begin{enumerate}
\item  Create the vertices $\{s_j^c,t_j^c, u_j, w_j, w'_j\}$.
\item Add the edges in $\{(s_j^c, u_j)$, $(w_j, w'_j)$, $ (w'_j, t_j^c)\}$.
\item If $x_{3j-2} = v_{i_j}$, add the edges $(u_j, p_{3j-2}^{\text{in}})$ and $(p_{3j-2}^{\text{out}}, w_j)$. Otherwise, add the edges $(u_j, n_{3j-2}^{\text{in}})$ and $(n_{3j-2}^{\text{out}}, w_j)$.
\item If $x_{3j-1}= v_{i_j'}$, add the edges $(u_j, p_{3j-1}^{\text{in}})$ and $(p_{3j-1}^{\text{out}}, w_j)$. Otherwise, add the edges $(u_j, n_{3j-1}^{\text{in}})$ and $(n_{3j-1}^{\text{out}}, w_j)$.
\item If $x_{3j}= v_{i_j''}$, add the edges $(s_j^c, p_{3j}^{\text{in}})$ and $(p_{3j}^{\text{out}}, w'_j)$. Otherwise, add the edges $(s_j^c, n_{3j}^{\text{in}})$ and $(n_{3j}^{\text{out}}, w'_j)$.
\end{enumerate}
In what follows, we refer to the edges and vertices that get added in the aforementioned 5-step construction relative to a given $C_j$ as the clause gadget for $C_j$. For each clause $C_j=(x_{3j-2}\vee x_{3j-1}\vee x_{3j})$, there are three directed paths from $s_j^c$ to $t_j^c$ each of which corresponds to one of the three literals in $C_j$. For example, for the first literal $x_{3j-2}$, there is a directed path from $s_j^c$ to $t_j^c$ that intersects with the edge $(p_{3j-2}^{\text{in}},p_{3j-2}^{\text{out}})$ on $\pi_{i_j}^+$ if $x_{3j-2}=v_{i_j}$ and that intersects with the edge $(n_{3j-2}^{\text{in}},n_{3j-2}^{\text{out}})$ on $\pi_{i_j}^-$  if $x_{3j-2}=\bar{v}_{i_j}$. To illustrate, assume that $x_{3j-2}=v_{i_j}$, $x_{3j-1}=\bar{v}_{i_j'}$, and $x_{3j}=v_{i_j''}$. For this specific case, the clause gadget for $C_j$ is shown on the right-hand side of Figure \ref{fig:disjointPaths}.  

Now, let $G$ be the directed graph that results from the construction of all variable and all clause gadgets. Observe that $G$ is acyclic. We next set up an instance $I'$ of {\sc $\forall \exists$ Phylo-Directed-Dis\-joint-Connecting-Paths}. Let $\cT$ be the caterpillar $(\ell_1^v,\ell_2^v,\ldots,\ell_n^v,\ell_{1}^c,\ell_{2}^c,\ldots,\ell_{m}^c)$. We  obtain a directed acyclic graph  $\cN$ from $\cT$ and $G$ by identifying $\ell_i^v$ with $s_i^v$ for each  $i\in\{1,2,\ldots,n\}$ and identifying $\ell_j^c$ with $s_j^c$ for each $j\in\{1,2,\ldots,m\}$. Clearly, $\cN$ is connected and has no parallel edges. Moreover, except for the root, since each vertex of $G$ has in-degree one and out-degree two, in-degree two and out-degree one, or in-degree one and out-degree zero, it follows that $\cN$ is a phylogenetic network on $T=\{t_1^v,t_2^v,\ldots,t_n^v,t_1^c,t_2^c,\ldots,t_m^c\}$. 
Let $S=\{s_1^v,s_2^v,\ldots,s_n^v,s_1^c,s_2^c,\ldots,s_m^c\}$. 
Since every vertex of $G$ that is not contained in $S$ lies on a directed path from a child of a vertex in $S$ to a leaf in $\cN$, it follows that $\cN$ is caterpillar-inducing with respect to $S$. Moreover, for each $i\in\{1,2,\ldots,n\}$, there are exactly two directed paths from $s_i^v$ to $t_i^v$ in $G_i^v$ and, hence, in $\cN$ that only intersect in the vertices $s_i^v$, $t_i^v$, and $y_i$, and the edge $(y_i,t_i^v)$. Recalling that $1\leq p<n$, it  follows from the construction that $\cN$ has the two-path property relative to $p$, and that both $P^\forall$ and $P^\exists$ are non-empty. We now set 
\begin{align*}
P^\forall &= \{(s_1^v,t_1^v),(s_2^v,t_2^v),\ldots,(s_p^v,t_p^v)\} \text{ and }\\  
P^\exists &= \{(s_{p+1}^v,t_{p+1}^v),(s_{p+2}^v,t_{p+2}^v),\ldots,(s_{n}^v,t_{n}^v)\} \cup \{(s_1^c, t_1^c),(s_2^c, t_2^c),\ldots, (s_m^c, t_m^c)\}.
\end{align*}
This completes the description of $I'$.

Since the number of vertices of $G$ is $3n+11m$, the number of vertices of $\cT$ is $2(n+m)-1$, and $G$ and $\cT$ have $n+m$ vertices in common, it follows that $\cN$ has size  $O(n+m)$ and can be constructed in polynomial time.

We complete the proof by establishing the following sublemma.

\begin{sublemma}\label{slem:iffPaths}
The instance $I$ is a yes-instance if and only if the instance $I'$ is a yes-instance.
\end{sublemma}

\begin{proof}
First, suppose that $I$ is a yes-instance. Let $\Pi^\forall=\{\pi_1^v,\pi_2^v,\ldots,\pi_p^v\}$ be a set of directed paths in $\cN$ such that each $\pi_i^v$ begins at $s_i^v$ and ends at $t_i^v$. As $p<n$, we have $\pi_i^v\in\{\pi_i^+,\pi_i^-\}$. Moreover, since $\cN$ has the two-path property relative to $p$, the paths in $\Pi^\forall$ are mutually vertex disjoint in $\cN$. Now, let $\beta: V\rightarrow\{F,T\}$ be a truth assignment that satisfies each clause of $\Psi$ such that, if $\pi_i^v=\pi_i^+$, then $v_i=F$ and, otherwise, $v_i=T$ for each $i\in\{1,2,\ldots,p\}$. Since $I$ is a yes-instance, $\beta$ exists. We next construct a directed path for each pair of vertices in $P^\exists$ such that, collectively, these paths together with the elements in $\Pi^\forall$ form a solution to $I'$. For each $i\in\{p+1,p+2,\ldots,n\}$, set $\pi_i^v=\pi_i^+$ if $v_i=F$ and set $\pi_i^v=\pi_i^-$ if $v_i=T$. Furthermore, for each $j\in\{1,2,\ldots,m\}$, let $x_{j'}$, with $j'\in\{3j-2,3j-1,3j\}$, be a literal in $C_j$ that is satisfied by $\beta$, and let $i$ be the element in $\{1,2,\ldots,n\}$ such that $x_{j'}\in\{v_i,\bar{v}_i\}$. By construction of the clause gadget, there is a directed path, say $\pi_j^c$, from $s_j^c$ to $t_j^c$ in $\cN$ such that one of the following properties applies.
\begin{enumerate}[(i)]
\item If $x_{j'}=v_{i}$, then $\pi_j^c$ contains the edge $(p_{j'}^\text{in},p_{j'}^\text{out})$.
\item If $x_{j'}=\bar{v}_{i}$, then $\pi_j^c$ contains the edge $(n_{j'}^\text{in},n_{j'}^\text{out})$.
\end{enumerate}
In Case (i), as $v_i=T$, we have $\pi_i^v=\pi_i^-$, and it follows that $\pi_j^c$ does not intersect $\pi_i^v$. Similar in Case (ii), as $v_i=F$, we have $\pi_i^v=\pi_i^+$, and it again follows that $\pi_j^c$ does not intersect $\pi_i^v$. By construction of $\cN$, it is now straightforward to check that \[\Pi^\forall\cup\{\pi_{p+1}^v,\pi_{p+2}^v,\ldots,\pi_{n}^v,\pi_1^c,\pi_2^c,\ldots,\pi_m^c\}\] is a collection of mutually vertex-disjoint directed-paths in $\cN$ that connect each pair of vertices in $P^\forall\cup P^\exists$. In particular, since the argument presented in this paragraph applies to all choices of directed paths in $\Pi^\forall$, we conclude that $I'$ is a yes-instance.

Second, suppose that $I'$ is a yes-instance. Let $\beta^\forall: \{v_1,v_2,\ldots,v_p\}\rightarrow\{F,T\}$ be a truth assignment. Furthermore, let \[\Pi=\{\pi_1^v,\pi_2^v,\ldots,\pi_p^v\}\cup\{\pi_{p+1}^v,\pi_{p+2}^v,\ldots,\pi_{n}^v,\pi_1^c,\pi_2^c,\ldots,\pi_m^c\}\] be a collection of mutually vertex-disjoint directed paths in $\cN$ such that $\pi_i^v=\pi_i^-$ if $v_i=T$ and $\pi_i^v=\pi_i^+$ if $v_i=F$ for each $i\in\{1,2,\ldots,p\}$. 
Since $I'$ is a yes-instance, $\Pi$ exists. Now, let $\beta: V \rightarrow\{F,T\}$  such that
\begin{enumerate}[(i)]
\item for each $i\in\{1,2,\ldots,p\}$, we have $\beta(v_i)=\beta^\forall(v_i)$ and, 
\item for each $i\in\{p+1,p+2,\ldots,n\}$, we have $\beta(v_i)=F$ if $\pi_i^v=\pi_i^+$ and, $\beta(v_i)=T$ if $\pi_i^v=\pi_i^-$. 
\end{enumerate}
We next show that $\beta$ satisfies each clause of $\Psi$. Let $C_j=(x_{3j-2}\vee x_{3j-1}\vee x_{3j})$ be a clause of $\Psi$ with $j\in\{1,2,\ldots,m\}$. Consider the directed path $\pi_j^c\in\Pi$ from $s_j^c$ to $t_j^c$ in $\cN$. Let $j'$ be the unique element in $\{3j-2,3j-1,3j\}$ such that $\pi_j^c$ contains either the edge $(p_{j'}^\text{in},p_{j'}^\text{out})$ or the edge $(n_{j'}^\text{in},n_{j'}^\text{out})$, and let $i$ be the element in $\{1,2,\ldots,n\}$ such that $x_{j'}\in\{v_i,\bar{v}_i\}$.
First, assume that $\pi_j^c$ contains $(p_{j'}^\text{in},p_{j'}^\text{out})$. Then, as $x_{j'}=v_i$ and the paths in $\Pi$ are mutually vertex disjoint in $\cN$, it follows that $\pi_i^v=\pi_i^-$. Hence $\beta(v_i)=T$. Second, assume that $\pi_j^c$ contains $(n_{j'}^\text{in},n_{j'}^\text{out})$. Then, as $x_{j'}=\bar{v}_i$ and the paths in $\Pi$ are mutually vertex disjoint, it follows that $\pi_i^v=\pi_i^+$. Hence $\beta(v_i)=F$. Under both assumptions, $\beta$ satisfies $C_j$ because $\beta(x_{j'})=T$. It now follows that $\beta$ satisfies $\Psi$ and, as the argument applies to all choices of truth assignments for the elements in $\{v_1,v_2,\ldots,v_p\}$, we conclude that $I$ is a yes-instance.  
\end{proof}
This completes the proof of Theorem~\ref{t:disjoint-paths}.
\end{proof}

While the next corollary is not needed for the remainder of the paper, it may be of independent interest in the theoretical computer science community.

\begin{corollary}
The decision problem {\sc $\forall \exists$ Directed-Disjoint-Connect\-ing-Paths} is $\Pi_2^P$-complete. 
\end{corollary}

\begin{proof}
Since every instance of {\sc $\forall \exists$ Phylo-Directed-Disjoint-Connect\-ing-Paths}  is also an instance of {\sc $\forall \exists$ Directed-Disjoint-Connecting-Paths}, it follows from Theorem~\ref{t:disjoint-paths} that the latter problem is $\Pi_2^P$-hard. To establish that {\sc $\forall \exists$ Directed-Disjoint-Connecting-Paths} is in $\Pi_2^P$, we use the same argument as in the first paragraph of the proof of Theorem~\ref{t:disjoint-paths} and, additionally, check in polynomial time if the paths in $\Pi^\forall$ are vertex disjoint.
\end{proof}

\subsection{{\sc Display-Set-Containment} is $\Pi_2^P$-complete}\label{sec:display-set-containment}
In this section, we show that {\sc Display-Set-Containment} is complete for the second level of the polynomial hierarchy. This problem is a generalization of the well-known NP-complete {\sc Tree-Containment} problem~\cite{kanj08}. 

\begin{theorem}\label{t:display-containment}
{\sc Display-Set-Containment} is $\Pi_2^P$-complete. 
\end{theorem}

\begin{proof}
We first show that {\sc Display-Set-Containment} is in $\Pi_2^P$. Let $\cN$ and $\cN'$ be two phylogenetic networks on $X$. To decide if $T(\cN)\subseteq T(\cN')$, guess a switching of $\cN$. Let $\cT$ be the phylogenetic $X$-tree that is yielded by $S$. Then use an NP-oracle for {\sc Tree-Containment} to decide if $\cT$ is displayed by $\cN'$. Since $\cN$ and $\cN'$ form a no-instance precisely if there exists some switching for $\cN$ that yields a phylogenetic tree that is not displayed by $\cN'$, it follows that {\sc Display-Set-Containment} is in co-NP$^{\text{NP}}=\Pi_2^P$. 

To complete the proof, we establish a reduction from {\sc $\forall \exists$ Phylo-Di\-rect\-ed-Disjoint-Connecting-Paths}. Using the same notation as in the formal statement of {\sc $\forall \exists$ Phylo-Di\-rect\-ed-Disjoint-Connecting-Paths}, let $I$ be the following instance of this problem. Let $\cN$ be a phylogenetic network on  $X$, let $S=\{s_1,s_2,\ldots,s_k\}$ and $T=X=\{t_1,t_2,\ldots,t_k\}$ be two  disjoint sets  of vertices of $\cN$, and let $p$ be an integer with $1\leq p < k$ such that $\cN$ is caterpillar-inducing with respect to $S$ and has the  two-path property relative to $p$. Furthermore, let
\begin{align*}
P^\forall &= \{(s_1,t_1),(s_2,t_2),\ldots,(s_p,t_p)\}, \\
P^\exists &= \{(s_{p+1},t_{p+1}),(s_{p+2},t_{p+2}),\ldots,(s_k,t_k)\}
\end{align*}
be two collections of pairs of elements in $S$ and $T$. This completes the description of $I$.  

\begin{figure}
\centering
\includegraphics[width=\textwidth]{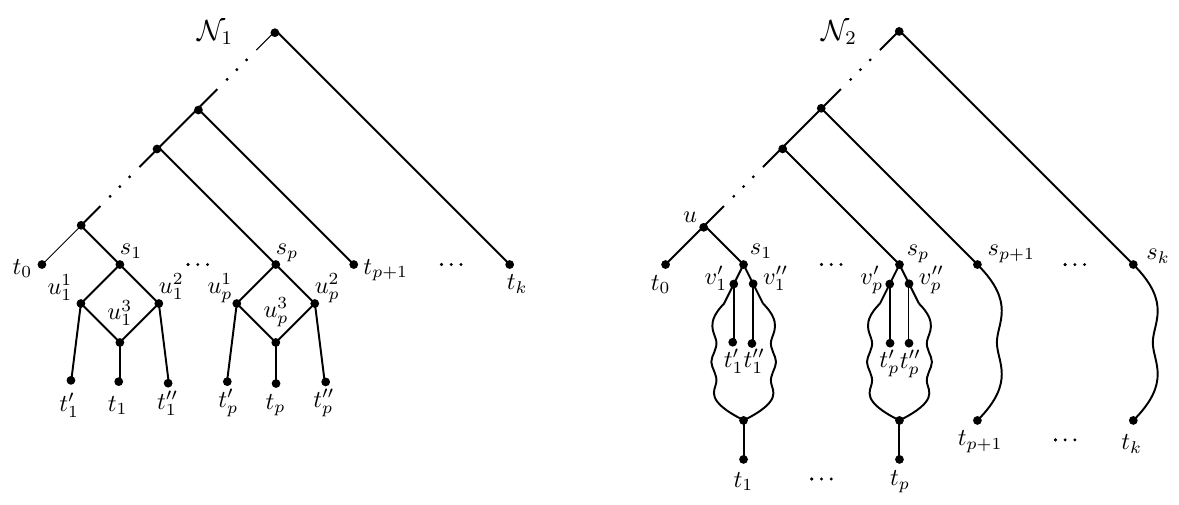}
\caption{The two phylogenetic networks $\cN_1$ and $\cN_2$ that are constructed in the proof of Theorem~\ref{t:display-containment}. For reasons of simplicity, not all edges of $\cN_2$ are shown. In particular, each squiggly line is a directed path and, depending on the given instance of {\sc $\forall \exists$ Phylo-Di\-rect\-ed-Disjoint-Connecting-Paths}, squiggly paths may intersect with each other and may be  further interconnected by paths that are not shown.}
\label{fig:gadgetTreesDisplayedSubset}
\end{figure}
Now, let $\cN _1$ be the phylogenetic network obtained from the caterpillar $(t_0,s_1,s_2,\ldots,s_p,t_{p+1},t_{p+2},\ldots,t_k)$ by adding the following edges and vertices for each $i\in\{1,2,\ldots,p\}$. Create three vertices $u_i^1$, $u_i^2$, and $u_i^3$ and  add the set \[\{(s_i,u_i^1),(s_i,u_i^2),(u_i^1,u_i^3),(u_i^2,u_i^3),(u_i^3,t_i),(u_i^1,t_i'),(u_i^2,t_i'')\}\] of edges. Observe that the leaf set of $\cN_1$ is \[X'=\{t_0,t_1,t_2,\ldots,t_k\}\cup\{t_i',t_i'': i\in\{1,2,\ldots,p\}\}.\]
The construction of $\cN_1$ is shown on the left-hand side of Figure~\ref{fig:gadgetTreesDisplayedSubset}.
We complete the reduction to an instance of {\sc Display-Set-Containment} by describing a second phylogenetic network $\cN_2$. For each $i\in\{1,2,\ldots,p\}$, let $w_i'$ and $w_i''$ be the two children of $s_i$ in $\cN$. As $\cN$ has the two-path property relative to $p$, recall that there are exactly two directed paths from $s_i$ to $t_i$ in $\cN$, and these two paths only have $s_i$, $t_i$, and the parent of $t_i$ in common. In the remainder of the proof, we denote the directed path from $s_i$ to $t_i$ that contains $w_i'$ with $\pi_i'$ and, similarly, we denote the directed path from $s_i$ to $t_i$ that contains $w_i''$ with $\pi_i''$. Lastly, we denote the parent of $s_1$ with $p_1$. Now, obtain $\cN_2$ from $\cN$ in the following way.
\begin{enumerate}[(i)]
\item Subdivide the edge $(p_1,s_1)$ with a new vertex $u$ and add the edge $(u,t_0)$.
\item For each $i\in\{1,2,\ldots,p\}$, subdivide $(s_i,w_i')$ with a new vertex $v_i'$, subdivide $(s_i,w_i'')$ with a new vertex $v_i''$, and add the two edges $(v_i',t_i')$ and $(v_i'',t_i'')$. 
\end{enumerate}
Clearly, the leaf set of $\cN_2$ is $X'$. To illustrate, $\cN_2$ is shown on the right-hand side in Figure~\ref{fig:gadgetTreesDisplayedSubset}.

As the size of $X'$ is polynomial in the size of $X$, it follows that the size of $\cN_1$ and $\cN_2$ is polynomial in the size of $\cN$. Furthermore, the construction of $\cN_1$ and $\cN_2$ takes polynomial time. 

\begin{sublemma}
The instance $I$ is a yes-instance if and only if $T(\cN_1)\subseteq T(\cN_2)$.
\end{sublemma}

\begin{proof}
First, suppose that $I$ is a yes-instance. Let $\cT'$ be a phylogenetic $X'$-tree that is displayed by $\cN_1$. For each $i\in\{1,2,\ldots,p\}$, note that $\cT'$ contains one of the two caterpillars $(t_i,t_i',t_i'')$ or $(t_i,t_i'',t_i')$. 
Let $\cJ'$ be the set that consists of each element $i\in\{1,2,\ldots,p\}$ for which $\cT'$ contains $(t_i,t_i',t_i'')$ and, similarly, let $\cJ''$ be the set that consists of each element $i\in\{1,2,\ldots,p\}$ for which $\cT'$ contains $(t_i,t_i'',t_i')$.
Furthermore, let $\Pi^\forall=\{\pi_1,\pi_2,\ldots,\pi_p\}$ be the set of directed paths in $\cN$ such that $\pi_i=\pi_i'$ if $i\in\cJ'$ and $\pi_i=\pi_i''$ if $i\in\cJ''$. 
Since $I$ is a yes-instance, there exists a set $\Pi=\Pi^\forall\cup\{\pi_{p+1},\pi_{p+2},\ldots,\pi_k\}$ of mutually vertex-disjoint directed paths in $\cN$, where $\pi_j$ is a directed path from $s_j$ to $t_j$ for each $j\in\{p+1,p+2,\ldots,k\}$. Moreover, as $\cN$ is caterpillar-inducing with respect to $S$, it is straightforward to check that there exists a  phylogenetic $X$-tree $\cT$ such that the following three properties are satisfied:
\begin{enumerate}[(i)]
\item $\cT$ is displayed by $\cN$,
\item $\cT=\cT'|X$, and
\item there exists an embedding of $\cT$ in $\cN$ that contains all edges of paths in $\Pi$. 
\end{enumerate}
Let $E_\cT$ be an embedding of $\cT$ in $\cN$ that satisfies (iii). By construction of $\cN_2$ from $\cN$, there exists an embedding of $\cT$ in $\cN_2$ whose set of edges is 
\begin{eqnarray*}
E_\cT'&=&(E_\cT-(\{(p_1,s_1)\}\cup\{(s_i,w_i'):i\in\cJ'\}\cup\{(s_i,w_i''):i\in\cJ''\}))\cup\\
&&\{(p_1,u),(u,s_1)\}\cup\{(s_i,v_i'),(v_i',w_i'):i\in\cJ'\}\cup\\
&&\{(s_i,v_i''),(v_i'',w_i''):i\in\cJ''\}.
\end{eqnarray*}
For each $i\in\{1,2,\ldots,p\}$, let $E_i'$ be the subset $\{(v_i',t_i'),(v_i'',t_i''),(s_i,v_i'')\}$ of edges in $\cN_2$ if $i\in\cJ'$, and  the subset $\{(v_i'',t_i''),(v_i',t_i'),(s_i,v_i')\}$ of edges in $\cN_2$ if $i\in\cJ''$. Since $E_\cT'$ is an embedding of $\cT$ in $\cN_2$, it now follows that \[E_\cT'\cup E_1'\cup E_2' \cup\cdots\cup E_p'\cup \{(u,t_0)\}\] is an embedding of $\cT'$ in $\cN_2$. Hence, $T(\cN_1)\subseteq T(\cN_2)$.

Second, suppose that $I$ is a no-instance. Throughout this part of the proof, we use $\pi_i$ to denote a directed path from $s_i$ to $t_i$ in $\cN$ for each $i\in\{1,2,\ldots,k\}$. Then, as $\cN$ has the two-path property relative to $p$, there is a set  $\Pi^\forall=\{\pi_1,\pi_2,\ldots,\pi_p\}$ of mutually vertex-disjoint directed paths in $\cN$ for which every set $\Pi=\Pi^\forall\cup\{\pi_{p+1},\pi_{p+2},\ldots,\pi_k\}$ of directed paths in $\cN$ contains two elements that are not vertex disjoint. For each $i\in\{1,2,\ldots,k\}$, let $E_i$ be the set of edges of $\pi_i$ in $\cN$. Furthermore, for each $i\in\{1,2,\ldots,p\}$, let $E_i'$ be the subset \[(E_i-\{(s_i,w_i')\})\cup\{(s_i,v_i'),(v_i',w_i'),(v_i',t_i'),(s_i,v_i''),(v_i'',t_i'')\}\] of edges in $\cN_2$ if $\pi_i=\pi_i'$, and  the subset \[(E_i-\{(s_i,w_i'')\})\cup\{(s_i,v_i''),(v_i'',w_i''),(v_i'',t_i''),(s_i,v_i'),(v_i',t_i')\}\] of edges in $\cN_2$ if $\pi_i=\pi_i''$, where $\pi_i'$ or $\pi_i''$ are as described in the construction of $\cN_2$ from $\cN$. Clearly, there is a phylogenetic tree $\cT_p$ with leaf set $\{t_i,t_i',t_i'':i\in\{1,2,\ldots,p\}\}$  for which there exists an embedding in $\cN_2$ that contains all edges in $E_1'\cup E_2'\cup\cdots\cup E_p'$. Observe that $\cT_p$ can be obtained from the caterpillar $(\ell_1,\ell_2,\ldots,\ell_p)$ by replacing each $\ell_i\in\{\ell_1,\ell_2,\ldots,\ell_p\}$ with the caterpillar $(t_i,t_i',t_i'')$ if $\pi_i=\pi_i'$ and with the caterpillar $(t_i,t_i'',t_i')$ if $\pi_i=\pi_i''$.
By construction, it now follows that $\cN_1$ displays $\cT_p$. Let $\cT$ be the unique phylogenetic $X'$-tree that is displayed by $\cN_1$ such that $\cT|\{t_i,t_i',t_i'':i\in \{1,2,\ldots,p\}\}=\cT_p$. We complete the argument by showing that $\cT$ is not displayed by $\cN_2$. Towards a contradiction, assume that $\cT$ is displayed by $\cN_2$.  Let $E_\cT'$ be an embedding of $\cT$ in $\cN_2$.
Then, since $\cT$ contains $(t_i,t_i',t_i'')$ or $(t_i,t_i'',t_i')$ for each $i\in\{1,2,\ldots,p\}$ and $\cN$ satisfies the two-path property relative to $p$, it follows from the construction of $\cN_2$ that $E_\cT'$ contains all edges in $E_1'\cup E_2'\cup\cdots\cup E_p'$. 
Furthermore, observe that there is a unique directed path from the root, say $\rho$, of $\cN_2$ to $t_0$, and so the edges on this path are elements of $E_\cT'$. For each pair $i$ and $i'$ of distinct elements in $\{1,2,\ldots,k\}$, it therefore follows that the directed path from $\rho$ to $t_i$ in $E_\cT'$ and the directed path  from $\rho$ to $t_{i'}$ in $E_\cT'$ only intersect in vertices that are ancestors of $t_0$ in $\cN_2$. Hence, as $\cN_2$ is caterpillar-inducing with respect to $S$, there exist directed paths $\pi_1^*,\pi_2^*,\ldots,\pi_p^*,\pi_{p+1}^*,\ldots,\pi_{k}^*$ in $E_\cT'$  such that the following three properties are fulfilled.
\begin{enumerate}[(i)]
\item For each $i\in\{1,2,\ldots,p\}$, $\pi_i^*$ is the unique directed path from $s_i$ to $t_i$ in $\cN_2$ that contains $v_i'$ if $\pi_i=\pi_i'$ and that contains $v_i''$ if $\pi_i=\pi_i''$.
\item For each $i\in\{p+1,p+2,\ldots,k\}$, $\pi_i^*$ is a directed path from $s_i$ to $t_i$ in $\cN_2$.
\item The elements in $\Pi^*=\{\pi_1^*,\pi_2^*,\ldots,\pi_{k}^*,\}$ are mutually vertex disjoint.
\end{enumerate}
Now, by construction, observe that $\pi_i^*$ is also a directed path from $s_i$ to $t_i$ in $\cN$ for each $i\in\{p+1,p+2,\ldots,k\}$.
As $\Pi^*$ is a set of  mutually vertex-disjoint directed paths in $\cN_2$, it now follows that, $\Pi^\forall\cup\{\pi_{p+1}^*,\pi_{p+2}^*,\ldots, \pi_{k}^*\}$ is a set of mutually vertex-disjoint directed paths in $\cN$. In turn, this implies that $I$ is a yes-instance; a contradiction. Hence, $\cT\notin T(\cN_2)$, and so $T(\cN_1)\nsubseteq T(\cN_2)$. 
\end{proof}

This establishes Theorem~\ref{t:display-containment}.
\end{proof}

We end this section with a brief discussion of the structural properties of the phylogenetic network $\cN_1$ that is constructed in the proof of Theorem~\ref{t:display-containment}. These properties will play an important role in the next section when we establish $\Pi_2^P$-completeness of {\sc Display-Set-Equivalence}. Let $\cN$ be a phylogenetic network on $X$. We say that $\cN$ is a {\it caterpillar network} if it can be obtained from a caterpillar $(\ell_1,\ell_2,\ldots,\ell_k)$ with $2\leq k\leq |X|$ by replacing each $\ell_i$ with a  phylogenetic network $\cN_i$ on $X_i$ such that the elements in $\{\cN_1,\cN_2,\ldots,\cN_k\}$ are pairwise vertex disjoint and 
\[\bigcup_{i=1}^k X_i=X.\] 
By construction, $\cN_1$ is a caterpillar network. Moreover, it is easily seen that $\cN_1$ is temporal and tree-child.

The next corollary now immediately follows from Theorem~\ref{t:display-containment}. 

\begin{corollary}\label{cor:display-containment}
Let $\cN_1$ be a temporal tree-child caterpillar network on $X$, and let $\cN_2$ be a phylogenetic network on $X$. Then deciding whether $T(\cN_1)\subseteq T(\cN_2)$ is $\Pi_2^P$-complete.
\end{corollary}

\subsection{{\sc Display-Set-Equivalence} is $\Pi_2^P$-complete}\label{sec:display-set-equivalence}

With the result of Corollary~\ref{cor:display-containment} in hand, we are now in a position to establish the main result of Section~\ref{sec:DSE} which is the following theorem.

\begin{theorem}\label{t:display-equivalence}
{\sc Display-Set-Equivalence} is $\Pi_2^P$-complete. 
\end{theorem}

\begin{proof}
Let $\cN$ and $\cN'$ be two phylogenetic networks on $X$. By Theorem~\ref{t:display-containment}, the problem of deciding whether or not $T(\cN)\subseteq T(\cN')$ is in $\Pi_2^P$. Similarly, the problem of deciding whether or not $T(\cN')\subseteq T(\cN)$ is in $\Pi_2^P$. Hence, {\sc Display-Set-Equivalence} is in $\Pi_2^P$.

We next establish a polynomial-time reduction from {\sc Display-Set-Con\-tainment} to {\sc Display-Set-Equivalence}. Let $\cN_1$ and $\cN_2$ be two phylogenetic networks on $X=\{\ell_1,\ell_2,\ldots,\ell_n\}$ that form the input to an instance of {\sc Display-Set-Containment} that asks if $T(\cN_1)\subseteq T(\cN_2)$. By Corollary~\ref{cor:display-containment}, we may assume that $\cN_1$ is a caterpillar network. Then there exist two vertex-disjoint phylogenetic networks $\cM_1$ and $\cM_{1'}$ with  leaf sets $W_1$ and $W_{1'}$, respectively, such that $W_1\cup W_{1'}=X$, and $\cN_1$ can be obtained from the caterpillar $\{x_1,x_2\}$ by replacing $x_1$ with $\cM_1$ and $x_2$ with $\cM_{1'}$. To ease reading, let $\cN_1'$ and $\cN_2'$ be the two phylogenetic networks on $X'=\{\ell_1',\ell_2',\ldots,\ell_n'\}$ that are obtained from $\cN_1$ and $\cN_2$, respectively, by replacing $\ell_i$ with $\ell_i'$ in both  networks for each $i\in\{1,2,\ldots,n\}$. Similarly, let $\cM_1'$ and $\cM_{1'}'$ be the two phylogenetic networks obtained from $\cM_1$ and $\cM_{1'}$, respectively, by replacing $\ell_i$ with $\ell_i'$ in exactly one of $\cM_1$ and $\cM_{1'}$ for each $i\in\{1,2,\ldots,n\}$. If $W_1'$ (resp. $W_{1'}'$) denotes the leaf set of $\cM_1'$ (resp. $\cM_{1'}'$), then $W_1'\cup W_{1'}'=X'$.

\begin{figure}[t]
\center
\includegraphics[width=\textwidth]{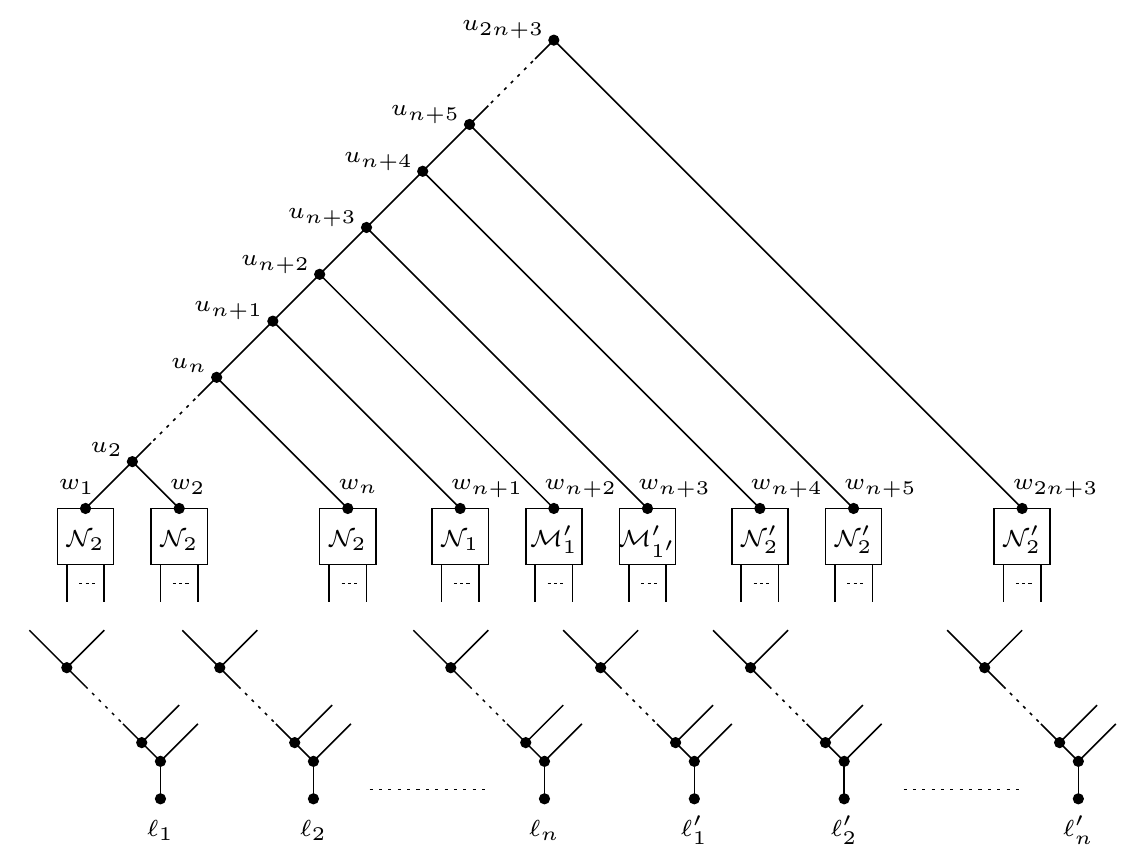}
\caption{The phylogenetic network $\cN_1^*$ on $X\cup X'$ as constructed in the proof of Theorem~\ref{t:display-equivalence}. Each of the squares labeled $\cN_2$, $\cN_1$, $\cM_1'$, $\cM_{1'}'$, and $\cN_2'$ refers to the network obtained from its namesake phylogenetic network by deleting all leaf labels. The dangling edges of each square are paired up with the edges shown in the bottom part of the figure as described in Step~(\ref{step-x}) of the construction in the proof of Theorem~\ref{t:display-equivalence}.}
\label{fig:DSE-network1}
\end{figure}

\begin{figure}[t]
\center
\includegraphics[width=\textwidth]{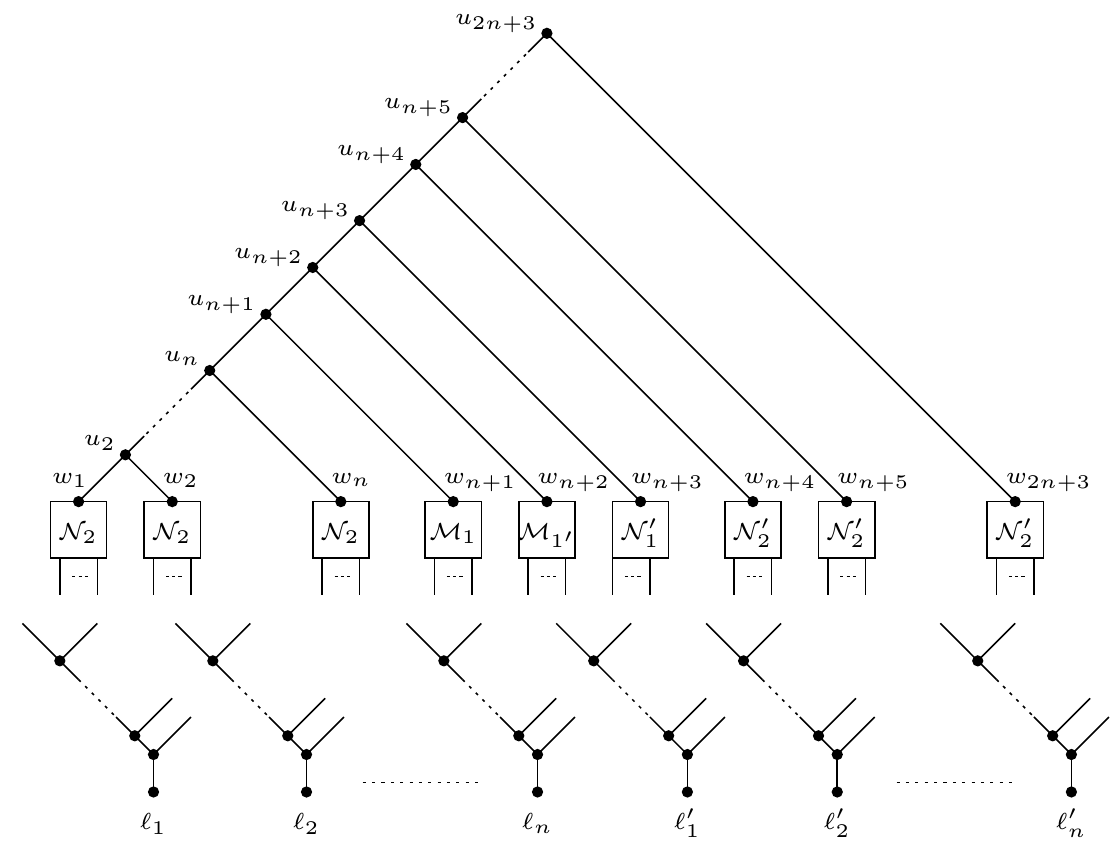}
\caption{The phylogenetic network $\cN_2^*$ on $X\cup X'$ as constructed in the proof of Theorem~\ref{t:display-equivalence}. Each of the squares labeled $\cN_2$, $\cM_1$, $\cM_{1'}$, $\cN_1'$ and $\cN_2'$ refers to the network obtained from its namesake phylogenetic network by deleting all leaf labels. The dangling edges of each square are paired up with the edges shown in the bottom part of the figure as described in Step~(\ref{step-x}) of the construction in the proof of Theorem~\ref{t:display-equivalence}.}
\label{fig:DSE-network2}
\end{figure}

Set $\cT$ as well as $\cT'$ to be the caterpillar $(w_1,w_2,\ldots,w_{2n+3})$. Furthermore, let $u_{2n+3},u_{2n+2},\ldots,u_2$ be the directed path in $\cT$ (and $\cT'$) such that, for all $j\in\{2,3,\ldots,2n+3\}$, $u_j$ is the parent of $w_j$. Now, let  $G_1^*$ and $G_2^*$ be the two directed acyclic graphs that are obtained from $\cT$ and $\cT'$, respectively, by applying the following six-step process.
\begin{enumerate}
\item For all $j\in\{1,2,\ldots,n\}$, replace $w_j$ with $\cN_2$ in $\cT$ and $\cT'$ by identifying $w_j$ with the root of $\cN_2$.
\item Replace $w_{n+1}$ with the root of $\cN_1$ in $\cT$ by identifying $w_{n+1}$ with the root of $\cN_1$, and replace $w_{n+1}$ with the root of $\cM_1$ in $\cT'$ by identifying $w_{n+1}$ with the root of $\cM_1$
\item Replace $w_{n+2}$ with $\cM_1'$ in $\cT$ by identifying $w_{n+2}$ with the root of $\cM_1'$, and replace $w_{n+2}$ with $\cM_{1'}$ in $\cT'$ by identifying $w_{n+2}$ with the root of $\cM_{1'}$
\item Replace $w_{n+3}$ with $\cM_{1'}'$ in $\cT$ by identifying $w_{n+3}$ with the root of $\cM_{1'}'$, and replace $w_{n+3}$ with $\cN_1'$ in $\cT'$ by identifying $w_{n+3}$  with the root of $\cN_1'$.
\item For all $j\in\{n+4,n+5,\ldots,2n+3\}$, replace $w_j$ with $\cN_2'$ in $\cT$ and $\cT'$ by identifying $w_j$ with the root of $\cN_2'$.
\item \label{step-x}For each $i\in\{1,2,\ldots,n\}$, identify all leaves labeled $\ell_i$ (resp. $\ell_i'$) in $\cT$ with a new vertex $v_i$ (resp. $v_i'$), add a new edge $(v_i,\ell_i)$ (resp. $(v_i',\ell_i')$). Do the same for all leaves labeled $\ell_i$ (resp. $\ell_i'$) in $\cT'$.
\end{enumerate}
To complete the construction, let $\cN_1^*$ and $\cN_2^*$ be two phylogenetic networks such that $G_1^*$ and $G_2^*$ can be obtained from $\cN_1^*$ and $\cN_2^*$, respectively, by contracting edges. Clearly, the leaf set of $\cN_1^*$ and $\cN_2^*$ is $X\cup X'$. Moreover, the directed path $u_{2n+3},u_{2n+2},\ldots,u_2$ of $\cT$ and $\cT'$ is also a directed path of $\cN_1^*$ and $\cN_2^*$. We refer to this path as the {\it backbone} of $\cN_1^*$ and $\cN_2^*$. The phylogenetic networks $\cN_1^*$ and $\cN_2^*$ are shown in Figures~\ref{fig:DSE-network1} and~\ref{fig:DSE-network2}, respectively. Lastly, observe that the size of both $\cN_1^*$ and $\cN_2^*$ is $O(n(|E_1|+|E_2|))$, where $E_1$ and $E_2$ is the edge set of $\cN_1$ and $\cN_2$, respectively. Hence, the construction of $\cN_1^*$ and $\cN_2^*$ takes polynomial time. 

\begin{sublemma}
$T(\cN_1)\subseteq T(\cN_2)$ if and only if $T(\cN_1^*)=T(\cN_2^*)$.
\end{sublemma}

\begin{proof}
Throughout this proof, let $U=\{u_2,u_3,\ldots,u_{2n+3}\}$ be the vertex set of the backbone of $\cN_1^*$ and $\cN_2^*$, and let \[E_U=\{(u_2,w_1),(u_2,w_2),(u_3,w_3),\ldots,(u_{2n+3},w_{2n+3})\}\] be the set of edges in $\cN_1^*$ and $\cN_2^*$ that are directed from a vertex in $U$ to a vertex not in $U$. Furthermore, for a vertex $v$ and an embedding $E$, we say that $v$ is {\it in} $E$ if there exists an edge in $E$ that is incident with $v$. If $v$ is in $E$, then we denote this by $v\in E$.

First, suppose that $T(\cN_1)\nsubseteq T(\cN_2)$. Let $\cT_1$ be a phylogenetic $X$-tree such that $\cT_1\in T(\cN_1)$ and $\cT_1\notin T(\cN_2)$. Let $\cT_1'$ be the phylogenetic $X'$-tree obtained from $\cT_1$ by replacing $\ell_i$ with $\ell_i'$ for each $i\in\{1,2,\ldots,n\}$. Furthermore, let $\cT$ be the phylogenetic $(X\cup X')$-tree obtained from $\cT_1$ and $\cT_1'$ by creating a new vertex $\rho$, adding an edge that joins $\rho$ with the root of $\cT_1$, and adding an edge that joins $\rho$ with the root of $\cT_1'$. As $\cN_1$ displays $\cT_1$ and $\cN_1'$ displays $\cT_1'$, it is easy to check that an embedding of $\cT$ in $\cN_2^*$ can be obtained from 
adding edges of
$\cN_2^*$ to \[\{(u_{n+3},u_{n+2}),(u_{n+2},u_{n+1}),(u_{n+1},w_{n+1}),(u_{n+2},w_{n+2}),(u_{n+3},w_{n+3})\}\] such that each element in $X$ is a descendant of $u_{n+2}$, each element in $X'$ is a descendant of $w_{n+3}$. Hence, $\cT$ is displayed by $\cN_2^*$.

We next show that $\cT$ is not displayed by $\cN_1^*$. Towards a contradiction, assume that $\cT$ is displayed by $\cN_1^*$. Let $E_1$ be an embedding of $\cT$ in $\cN_1^*$.
Furthermore, let $k$ be the maximum element in $\{1,2,\ldots,2n+3\}$ such that $w_k \in E_1$. By construction of $\cT$, either each element in $X$ is a descendant of $w_k$ in $E_1$ or each element in $X'$ is a descendant of $w_k$ in $E_1$. Thus, as $\cN_2$ does not display $\cT_1$ and $\cN_2'$ does not display $\cT_1'$, we have $k=n+1$. In particular, each element in $X$ is a descendant of $w_k$ in $E_1$. But no element in $X'$ is a descendant of $u_k$ in $E_1$; a contradiction. Hence, $\cT$ is not displayed by $\cN^*_1$, and so $T(\cN^*_1)\neq T(\cN^*_2)$.

Second, suppose that $T(\cN_1)\subseteq T(\cN_2)$. Let $\cT$ be a phylogenetic $(X\cup X')$-tree that is displayed by $\cN_1^*$, and let $E_1$ be  an embedding of $\cT$ in $\cN_1^*$. For each $j\in\{1,2,\ldots,2n+3\}$ with $w_j\in E_1$, let $Y_j$ be the set that consists of all leaves that are descendants of $w_j$ in $E_1$, and let $\cT_j$ be the phylogenetic tree obtained from the minimal rooted subtree of $E_1$ that connects all leaves in $Y_j$ by suppressing all vertices with in-degree one and out-degree one.  If $w_{n+1}\in E_1$, then, by the pigeonhole principle, there exists an element $j\in\{1,2,\ldots,n\}$ such that $w_j\notin E_1$. Similarly, if $w_{n+3}\in E_1$, then there exists an element $j'\in\{n+4,n+5\ldots,2n+3\}$ such that $w_{j'}\notin E_1$. Without loss of generality, we may therefore assume by the construction of $\cN_1^*$ that $E_1$ satisfies the following property.

\noindent (P) If $w_{n+1}\in E_1$, then $w_n\notin E_1$ and, if $w_{n+3}\in E_1$, then $w_{n+4}\notin E_1$.

Recall that each tree in $T(\cN_1)$ is displayed by $\cN_2$, each tree in $T(\cM_1')$ is displayed by $\cN_1'$, and each tree in $T(\cM_{1'}')$ is displayed by $\cN_2'$. Hence, there exists a set $E_2$ of edges of $\cN_2^*$ such that the following conditions are satisfied.
\begin{enumerate}[(i)]
\item For each $j\in\{1,2,\ldots,n, n+4,n+5,\ldots,2n+3\}$, if $w_j\in E_1$, then $w_j$ is the root of a subtree in $E_2$ that is a subdivision of $\cT_j$.
\item If $w_{n+1}\in E_1$, then $w_n$ is the root of a subtree in $E_2$ that is a subdivision of $\cT_{n+1}$.
\item If $w_{n+3}\in E_1$, then $w_{n+4}$ is the root of a subtree in $E_2$ that is a subdivision of $\cT_{n+3}$.
\item If $w_{n+2}\in E_1$, then $w_{n+3}$ is the root of a subtree in $E_2$ that is a subdivision of $\cT_{n+2}$.
\end{enumerate}
Since $E_1$ satisfies (P), $E_2$ is well defined. Moreover, as $\cT$ is displayed by $\cN_1^*$, it now follows that there exists an embedding of $\cT$ in $\cN_2^*$ that contains all edges in $E_2$. Thus $T(\cN_1^*)\subseteq T(\cN_2^*)$.

Now, let $\cT$ be a phylogenetic $(X\cup X')$-tree that is displayed by $\cN_2^*$. To see that $\cT$ is displayed by $\cN_1^*$, we can use the same argument as the one to show that $T(\cN_1^*)\subseteq T(\cN_2^*)$ even thought the assumption that $T(\cN_1)\subseteq T(\cN_2)$ is not symmetric. In particular, we interchange the roles of $\cN_1^*$ and $\cN_2^*$ (and, consequently, the roles of $E_1$ and $E_2$). Moreover, as each tree in $T(\cM_1)$ is displayed by $\cN_2$, each tree in $T(\cM_{1'})$ is displayed by $\cN_1$, and each tree in $T(\cN_1')$ is displayed by $\cN_2'$, only Condition (iv) above needs to be rewritten as follows.
\begin{enumerate}[(iv*)]
\item If $w_{n+2}\in E_2$, then $w_{n+1}$ is the root of a subtree in $E_1$ that is a subdivision of $\cT_{n+2}$.
\end{enumerate}
It is now straightforward to check that $\cT$ is displayed by $\cN_1^*$, and so $T(\cN_2^*)\subseteq T(\cN_1^*)$. Combining both cases establishes that $T(\cN_2^*)=T(\cN_1^*)$.
\end{proof}
This completes the proof of  Theorem~\ref{t:display-equivalence}.
\end{proof}

\section{Conclusion}\label{sec:conclu}
We end this paper, with three corollaries that are implied by the results presented in Section~\ref{sec:CTC} and an open problem.

For two temporal networks $\cN$ and $\cN'$ on $X$, the authors of~\cite{linz13} showed that counting the number of elements in $T(\cN)\cap T(\cN')$ is \#P-complete. Since {\sc Common-Tree-Containment} is the decision version of computing $|T(\cN)\cap T(\cN')|$ and computational hardness of a decision problem implies computational hardness of the associated counting problem, the next corollary follows from Theorem~\ref{t:CTC}.

\begin{corollary}
Let $\cN$ and $\cN'$ be two temporal normal networks on $X$. Then counting the number of elements in $\cT(\cN)\cap \cT(\cN')$ is \#{\rm P}-complete.
\end{corollary} 

In 2015, Francis and Steel~\cite{francis15} introduced tree-based networks. A phylogenetic network $\cN$ on $X$ is {\em tree-based} if, up to suppressing vertices of in-degree one and out-degree one, $\cN$ displays a phylogenetic $X$-tree $\cT$ that can be obtained by only deleting reticulation edges, in which case, $\cT$ is a {\em base tree} of $\cN$. If $\cN$ is tree-based, it is well known that not every phylogenetic $X$-tree displayed by $\cN$ is a base tree. However, noting that each tree-child network is also a tree-based network, it is shown in~\cite{semple16}  that a phylogenetic tree $\cT$ is displayed by a tree-child network $\cN$ if and only if $\cT$ is a base tree of $\cN$. Hence, for two tree-child networks $\cN$ and $\cN'$, the problem of deciding whether or not $T(\cN)\cap T(\cN')\ne\emptyset$ is equivalent to deciding whether or  not $\cN$ and $\cN'$ have a common base tree. 

\begin{corollary}\label{cor:last}
Let $\cN$ and $\cN'$ be two tree-based networks on $X$. Then deciding if $\cN$ and $\cN'$ have a common base tree is {\rm NP}-complete.
\end{corollary} 

\begin{proof}
Let $S$ be a switching of $\cN$, and let $\cT$ be a phylogenetic $X$-tree. 
We say that $S$ is a {\em base-tree switching} if, for each non-leaf vertex $u$ in $\cN$ that is the parent of only reticulations, there exists an edge $(u, v)$ in $S$.
By the definition of a tree-based network it follows that $\cT$ is a base tree of $\cN$ if and only if there exists a base-tree switching $S$ of $\cN$ that yields $\cT$.  Now, let $S$ be a switching of $\cN$, and let $S'$ be a switching of $\cN'$.  If $S$ is a base-tree switching of $\cN$ and $S'$ is a base-tree switching of $\cN'$, and $S$ and $S'$ yield the same tree, then $\cN$ and $\cN'$ have a common base tree. Since it can be checked in polynomial time if $S$ (resp. $S'$) is a base-tree switching of $\cN$ (resp. $\cN'$), and if $S$ and $S'$ yield the same tree, it follows that deciding whether or not $\cN$ and $\cN'$ have a common base tree is in NP. The corollary now follows from Theorem~\ref{t:CTC}. 
\end{proof}

Lastly, using (ordinary) switchings instead of base-tree switching, ideas analogous to the ones described in the proof of Corollary~\ref{cor:last} can be used to show that {\sc Common-Tree-Containment} is in NP for two arbitrary phylogenetic networks. The next corollary is now an immediate consequence of Theorem~\ref{t:CTC}.

\begin{corollary}
{\sc Common-Tree-Containment} is {\rm NP}-complete for two arbitrary phylogenetic networks.
\end{corollary}

Now, let $C$ be a class of phylogenetic networks for which {\sc Tree-Contain\-ment} is solvable in polynomial time such as tree-child or, more generally, reticulation-visible networks~\cite{bordewich16,gunawan17, iersel10}. Furthermore, let $\cN$ and $\cN'$ be two networks in $C$. Then deciding if $\cT(\cN)=T(\cN')$  is in co-NP because, given a tree $\cT$ that is displayed by $\cN$ or $\cN'$, it can be checked in polynomial time, if $\cT$ is also displayed by the other network. If this is not the case, then $\cN$ and $\cN'$ form a no-instance of  {\sc Display-Set-Equivalence}. Whether {\sc Display-Set-Equivalence} for $\cN$ and $\cN'$ is co-NP-complete remains an open problem. Nevertheless, it is unlikely that {\sc Display-Set-Equivalence} for $\cN$ and $\cN'$ is 
$\Pi_2^P$-complete since a problem that is $\Pi_2^P$-complete and in co-NP would imply that co-NP=$\Pi_2^P$ which, in turn, would result in a collapse of the polynomial hierarchy to the first level.

\end{document}